\documentclass{amsart}

\usepackage{amssymb}
\usepackage{overpic}
\usepackage{graphics}
\usepackage{epsfig}
\usepackage{amsmath}
\usepackage{amsfonts}
\usepackage{paralist}

\newtheorem{theorem}{Theorem}[section]
\newtheorem{corollary}{Corollary}

\newtheorem{lemma}[theorem]{Lemma}
\newtheorem{proposition}{Proposition}

\newtheorem{example}{Example}
\theoremstyle{definition}
\newtheorem{definition}[theorem]{Definition}
\newtheorem{remark}{Remark}

\DeclareMathOperator{\sech}{sech}

\begin{document}

\title[Restrictions and Stability of Time-Delayed Dynamical Networks]{Restrictions and Stability of Time-Delayed Dynamical Networks}

%    author one information
\author{L. A. Bunimovich$^1$}
%\address{ABC Math Program and School of Mathematics, Georgia Institute of Technology, 686 Cherry Street, Atlanta, GA 30332, USA}
%\email{bunimovich@math.gatech.edu}

%    author one information
\author{B. Z. Webb$^2$}

%\address{ABC Math Program and School of Mathematics, Georgia Institute of Technology, 686 Cherry Street, Atlanta, GA 30332, USA}
%\email{bwebb@math.gatech.edu}



\maketitle

\begin{center}
$^{1}$ \smaller{ABC Math Program and School of Mathematics, Georgia Institute of Technology, 686 Cherry Street, Atlanta, GA 30332, USA}\\
$^{2}$ Department of Mathematics, 308 TMCB, Brigham Young University, Provo, UT 84602, USA\\
E-mail: bunimovih@math.gatech.edu and bwebb@math.byu.edu
\end{center}

\begin{abstract}
This paper deals with the global stability of time-delayed dynamical networks. We show that for a time-delayed dynamical network with non-distributed delays the network and the corresponding non-delayed network are both either globally stable or unstable. We demonstrate that this may not be the case if the network's delays are distributed. The main tool in our analysis is a new procedure of dynamical network restrictions. This procedure is useful in that it allows for improved estimates of a dynamical network's global stability. Moreover, it is a computationally simpler and much more effective means of analyzing the stability of dynamical networks than the procedure of isospectral network expansions introduced in \cite{BW2012}. The effectiveness of our approach is illustrated by applications to various classes of Cohen-Grossberg neural networks.
\end{abstract}

\section{Introduction}
The study of networks in nature, science, and technology is an extremely active area of research. Because of the structural complexity and considerable size of many networks a good deal of this research has been devoted to understanding the static features of these systems \cite{Albert02,Dorogovtsev03,Faloutsos99,Newman06,Porter09,Strogatz03,Watts99}. However, most real networks are dynamic. That is, each network element has an associated state that changes with time. Additionally, because of finite processing speeds and transmission of signals over distances, the dynamics of such networks are inherently time-delayed.

To investigate the dynamical properties of networks, a general approach was introduced in \cite{Afriamovich07}. The major idea in this work is that a network's dynamics can be analyzed in terms of three key features; (i) the internal \emph{local} dynamics of the network elements, (ii) the \emph{interactions} between the network elements, and (iii) the \emph{topology} or structure of the graph of interactions of the network. In \cite{Afriamovich07} it was shown that each of these network features (including the network topology) can be interpreted and investigated as a dynamical system.

In this paper we extend this mathematical formalism to investigate the dynamics of time-delayed dynamical networks. That is, we consider dynamical networks in which the network's interaction includes time-delays. In such systems it is therefore possible for network elements to interact over large as well as multiple time-scales, which can have a significant impact on the network's dynamics.

The particular type of dynamics we consider in this paper is whether a time-delayed dynamical network has an evolution that is \emph{stable}, i.e. whether a network has a globally attracting fixed point. Specifically, our goal will be to describe what type of time-delayed interactions stabilize a given set of local systems. Our first major task in this regard is to understand how transforming a network, by either adding or removing time-delays to its interaction, effects the network's stability. A second related task is to determine under what conditions such transformations can be used to gain new information about the original network.

With respect to this first task we note that by modifying an interaction's time-delays it is possible to destabilize a stable network or conversely, stabilize a network that has unstable dynamics. However, one of the main results in this paper is that if a time-delayed dynamical network is known to be stable then the undelayed version of this network will also have stable dynamics (see theorem \ref{theorem3}). That is, by removing the time-delays in a stable network's interaction we do not change the network's dynamics in any essential way.

Conversely, the addition of delays to a network's interaction may result in a network with unstable dynamics even if the original network has a globally attracting fixed point. In the case where the state of each network element $x_i$ depends on at most one of the previous states of any element $x_j$ we say the network has non-distributed delays and has distributed delays otherwise (see definition \ref{SD}). We show that if a time-delayed dynamical network is stable and has a non-distributed interaction then any change in the interaction's time-delays does not effect the network's stability (see theorem \ref{theorem3.45}.)

This result allows us to analyze the dynamics of any time-delayed network as a network without delays if it has a non-distributed interaction. This is advantageous from a computational point of view since it is significantly simpler to investigate the stability of an undelayed network than a network with delays.

With this theory in place we consider the class of dynamical networks known as Cohen-Grossberg neural networks \cite{MCohen1983} whose stability has received a considerable amount of attention. See for instance \cite{JCao2005,MGupta1994,CLi2010,LTao2011,SChena2009,LWang2005}. By applying our theory to such systems we are able to derive new criteria for the stability of both the delayed and undelayed versions of this class of networks.

To further apply this theory we note that one of the major obstacle in determining the dynamic behavior of a network (or high-dimensional system) is that the information needed to do so is spread throughout the network components. In \cite{BW2012} it has been shown that a network can be modified in a number of ways that preserve its dynamics while concentrating this network information. This concept of a \emph{dynamical network expansion} allows for improved estimates on whether a network has a globally attracting fixed point. However, this improved ability to determine a network's stability comes at the price of needing to analyze the stability of a higher-dimensional dynamical network.

By combining the theory of dynamical network expansions with our present results on time-delayed networks we are able to develop a new theory of \emph{dynamical network restrictions}. Such restrictions are lower-dimensional versions of a given network which have a less complicated structure of interactions. We show that if a dynamical network is stable then any restriction of this network inherits this stability (see theorem \ref{theorem5}).

To investigate the stability of a network via one of its restrictions we introduce the notion of a \emph{basic structural set} which is related to the idea of a non-distributed interaction. The major result in this direction is that if a network is restricted to a \emph{basic structural set} and the restricted network is known to have stable dynamics then the same holds for the original network (see theorem \ref{theorem6}).

As with network expansions, such restrictions allow for improved estimates of a dynamical network's global stability. However, the computational effort required to construct and analyze these restrictions is far less when compared with that of dynamical network expansions.

One reason for this simplicity is that dynamical network expansions essentially preserve the entire spectrum of the network. However, to analyze the global stability of a network one does not need to know its entire spectrum but rather only the network's spectral radius. As we have a relationship between the spectral radius of a network and the network restricted to a basic structural set we are able to analyze the stability of one in terms of the other although the two may have very different spectra.

The results and topics in this paper are organized as follows. Section 2 introduces the concept of a dynamical network and gives a sufficient condition under which such systems have a globally attracting fixed point. Section 3 then defines a time-delayed dynamical network and extends this stability criteria to networks with time-delays. Following this section 4 considers how changes to a network's time-delays effects the network's stability and introduces the concept of a non-distributed interaction.

The second half of the paper, consisting of sections 5 and 6, integrates the theory of dynamical network expansions in section 5 with the results of sections 3 and 4. The result, given in section 6, is the notion of a dynamical network restriction and that of a basic structural set, which are used in this section to obtain improved stability estimates of both delayed and undelayed dynamical networks.

Lastly, we note that each of the results in this paper is illustrated by examples. Moreover, although our procedure deals with matrices and is therefore linear in nature, it does not in fact have this restriction and is applicable without any modifications to general nonlinear dynamical systems.

\section{Dynamical Networks and Global Stability}

As mentioned in the introduction, dynamical networks or networks of interacting dynamical systems are composed of (i) local dynamical systems which have their own (local intrinsic) dynamics, (ii) interactions between these local systems (elements of the network), and (iii) the graph of interactions (topology of the network).
\subsection{Dynamical Networks}
Following the approach in \cite{Afriamovich07,BW2012}, dynamical networks are defined as follows. Let $i\in \mathcal{I}=\{1,\dots,n\}$ and $\varphi_i: X_i\rightarrow X_i$ be maps on the complete metric space $(X_i,d)$ where
\begin{equation}\label{eq0.1}
L_i=\sup_{x_i\neq y_i\in X_i}\frac{d(\varphi_i(x_i),\varphi_i(y_i))}{d(x_i,y_i)}<\infty.
\end{equation}
Let $(\varphi,X)$ denote the direct product of the local systems $(\varphi_i,X_i)$ over $\mathcal{I}$ on the complete metric space $(X,d_{max})$ where for $\textbf{x},\textbf{y}\in X$
$$d_{max}(\textbf{x},\textbf{y})=\max_{i\in\mathcal{I}}\{d(x_i,y_i)\}.$$

\begin{definition}\label{dynamicalnetwork}
A map $F:X \rightarrow X$ is called an \textit{interaction} if for every $j\in \mathcal{I}$ there exists a nonempty collection of indices $\mathcal{I}_j\subseteq\mathcal{I}$ and a continuous function
$$F_j:\bigoplus_{i\in \mathcal{I}_j} X_i\rightarrow X_j,$$
where the map $F$ is defined as follows:
$$F(\mathbf{x})_j=F_j(\textbf{x}|_{\mathcal{I}_j}), \ \ j\in \mathcal{I}, \ \ \text{and} \ \ \textbf{x}\in X.$$
The superposition $\mathcal{F}=F\circ \varphi$ generates the dynamical system $(\mathcal{F},X)$ which is a \textit{dynamical network}.
\end{definition}

Suppose $F$ satisfies the following Lipschitz condition for finite constants $\Lambda_{ij}\geq 0:$
\begin{equation}\label{eq2.3}
d\big(F_j(\textbf{x}|_{\mathcal{I}_j}),F_j(\textbf{y}|_{\mathcal{I}_j})\big)\leq \sum_{i\in \mathcal{I}_j} \Lambda_{ij} d(x_i,y_i)
\end{equation}
for all $\textbf{x},\textbf{y}\in X$ where $\textbf{x}|_{\mathcal{I}_j}$ is the restriction of $\mathbf{x}\in X$ to $\bigoplus_{i\in \mathcal{I}_j} X_i$.

The Lipschitz constants $\Lambda_{ij}$ in equation (\ref{eq2.3}) form the matrix $\Lambda\in\mathbb{R}^{n\times n}$ where each $\Lambda_{ij}\geq 0$ and where $\Lambda_{ij}=0$ if $i\notin\mathcal{I}_j$. For the dynamical network $(\mathcal{F},X)$ we call the matrix
$$\Lambda^T\cdot diag[L_1,\dots,L_n]=  \left( \begin{array}{cccc}
\Lambda_{11}L_1 & \dots & \Lambda_{n1}L_n \\
\vdots & \ddots & \vdots \\
\Lambda_{1n}L_1 & \dots & \Lambda_{nn}L_n \end{array} \right)$$
a \emph{stability matrix} of $(\mathcal{F},X)$.

\begin{definition}
The dynamical network $(\mathcal{F},X)$ has a \textit{globally attracting fixed point} $\tilde{\textbf{x}}\in X$ if for any $\textbf{x}\in X$, $$\displaystyle{\lim_{k\rightarrow\infty}d_{max}\big(\mathcal{F}^k(\textbf{x}),\tilde{\textbf{x}}\big)=0}.$$ If $(\mathcal{F},X)$ has a globally attracting fixed point we say it is \textit{globally stable}.
\end{definition}

The local systems $(\varphi,X)$ are said to be \textit{stable} if $(\varphi,X)$ has a globally attracting fixed point and are said to be \textit{unstable} otherwise. The interaction $F:X\rightarrow X$ is said to \textit{stabilizes} the local systems $(\varphi,X)$ if the local systems are unstable but the dynamical network $(\mathcal{F},X)$ has a globally attracting fixed point. If the local systems $(\varphi,X)$ are stable and $(\mathcal{F},X)$ has a globally attracting fixed point we say the interaction $F:X\rightarrow X$ \textit{maintains} the stability of $(\varphi,X)$.

One of the major goals of this paper is to find sufficient conditions under which an interaction stabilizes a set of local systems. For a matrix $A\in\mathbb{R}^{n\times n}$ let $\rho(A)$ denote its \textit{spectral radius}, i.e. if $\sigma(A)$ are the eigenvalues of $A$ then
$$\rho(A)=\max\{|\lambda|:\lambda\in\sigma(A)\}.$$
The following theorem is found in \cite{BW2012} (see theorem 2.5).

\begin{theorem}\label{stability}
Suppose $A$ is a stability matrix the dynamical network $(\mathcal{F},X)$. If $\rho(A)<1$ then the dynamical network $(\mathcal{F},X)$ has a globally attracting fixed point.
\end{theorem}

In other words, theorem \ref{stability} states that if $\rho(A)<1$ then the interaction $F$ stabilizes (or maintains the stability) of the local systems $(\varphi,X)$.

Suppose the maps $\varphi:X\rightarrow X$ and $F:X\rightarrow X$ are continuously differentiable and each $X_i\subseteq\mathbb{R}$ is a closed interval. If the constants
\begin{align}\label{eq2}
L_i&=\max_{\textbf{x}\in X}|\varphi^\prime_{i}(x_i)|<\infty\\
\Lambda_{ij}&=\max_{\textbf{x}\in X}|(DF)_{ji}(\textbf{x})|<\infty\label{eq2.01}
\end{align}
where $DF$ is the matrix of first partial derivatives of $F$ then we say $\mathcal{F}\in C^1_\infty(X)$.
For $\mathcal{F}\in C^1_\infty(X)$ the matrix $\Lambda^T\cdot diag[L_1,\dots,L_n]$ given by (\ref{eq2}) and (\ref{eq2.01}) is a stability matrix of $(\mathcal{F},X)$. As this matrix is unique then for any $\mathcal{F}\in C^1_\infty(X)$ we let $$\rho(\mathcal{F})=\rho(\Lambda^T\cdot diag[L_1,\dots,L_n]).$$

For matrices $A,B\in\mathbb{R}^{n\times n}$ we write $A\leq B$ if $A_{ij}\leq B_{ij}$ for each $1\leq i,j\leq n$. The matrix $A\in\mathbb{R}^{n\times n}$ is \emph{nonnegative} if the zero matrix $0\leq A$. If it is known that $0\leq A\leq B$ then $\rho(A)\leq\rho(B)$ (see \cite{Horn85} for instance). This can be used to show the following.

\begin{proposition}\label{proposition1}
Let $\mathcal{F}\in C^1_\infty(X)$. If $A$ is any stability matrix of $(\mathcal{F},X)$ then $\rho(\mathcal{F})\leq\rho(A)$.
\end{proposition}

\begin{proof}
Suppose $\mathcal{F}\in C^1_\infty(X)$ and that $A=\tilde{\Lambda}^T\cdot diag[\tilde{L}_1,\dots,\tilde{L}_n]$ is a stability matrix of $(\mathcal{F},X)$. Let $\textbf{e}_i$ be the $i$th standard basis vector of $\mathbb{R}^n$. For $\textbf{x}\in X$ let $h\neq 0$ such that $\textbf{y}=\textbf{x}+h\textbf{e}_i\in X$. Then by (\ref{eq2.3}),
$$d\big(F_j(\textbf{x}|_{\mathcal{I}_j}),F_j(\textbf{y}|_{\mathcal{I}_j})\big)\leq \tilde{\Lambda}_{ij} d(x_i,x_i+h)=\tilde{\Lambda}_{ij}|h|.$$
Hence, $|F(\textbf{x})_j-F(\textbf{y})_j|/|h|\leq\tilde{\Lambda}_{ij}$. Taking the limit as $h\rightarrow 0$ implies
$$|DF_{ji}(\textbf{x})|\leq\tilde{\Lambda}_{ij} \ \text{for all} \ \textbf{x}\in X.$$
Letting $\Lambda_{ij}=\max_{\textbf{x}\in X}|DF_{ji}(\textbf{x})|$ then $0\leq\Lambda\leq\tilde{\Lambda}$.

Similarly, one can show that $L_i=\max_{\textbf{x}\in X}|\varphi^\prime_{i}(x_i)|\leq \tilde{L}_i$. Hence, the matrix $\Lambda^T\cdot diag[L_1,\dots,L_n]\leq A$ implying $\rho(\mathcal{F})\leq\rho(A)$.
\end{proof}

For $\mathcal{F}\in C^1_\infty(X)$, proposition \ref{proposition1} implies that $\rho(\mathcal{F})$ is optimal for directly determining via theorem \ref{stability} whether $(\mathcal{F},X)$ is globally stable.

\subsection{Cohen-Grossberg Neural Networks}
Consider the following local system $(\varphi_i,\mathbb{R})$ given by
\begin{equation}\label{CGloc}
\varphi_i(x_i)=(1-\epsilon)x_i+c_i
\end{equation}
where $c_i,\epsilon\in\mathbb{R}$ and $1\leq i\leq n$. Notice that the local systems $(\varphi,\mathbb{R}^n)$ are stable if and only if $|1-\epsilon|<1$. The major question we consider is what kind of interaction stabilizes or maintains the stability of these local systems.

For the local systems (\ref{CGloc}) we are specifically interested in interactions of the form
\begin{equation}\label{CGint}
F_j(\textbf{x})=x_j+\sum_{j=1}^n W_{ij}\phi_i\Big(\frac{x_i-c_i}{1-\epsilon}\Big)
\end{equation}
where $\epsilon\neq 1$, $W\in\mathbb{R}^{n\times n}$ and $\phi_i:\mathbb{R}\rightarrow\mathbb{R}$ is any smooth sigmoidal function with Lipschitz constant $\mathcal{L}\geq 0$. The reason we study this particular interaction is that the dynamical network $(\mathcal{F},X)$ is then given by
\begin{equation}\label{CG}
\mathcal{F}_j(\mathbf{x})=(1-\epsilon)x_j+\sum_{i=1}^n W_{ij}\phi(x_i)+c_j
\end{equation}
which is a special case of the Cohen-Grossberg neural network in discrete time \cite{MCohen1983}.

For such neural networks the variable $x_i$ represents the \emph{activation} of the $i$-th neuron population. The function $\phi_i(x_i)$ describes how the neuron populations react to inputs. A typical example is the function $\phi_i(x_i)=\tanh(\mathcal{L} x_i)$, see \cite{LWang2005}. The matrix $W$ gives the interaction weights between each of the $i$-th and $j$-th neuron populations describing how the neurons are connected within the network. The constants $c_i$ indicate constant inputs from outside the system.

As mentioned in the introduction, there is considerable interest in determining stability conditions for such networks in general. See for instance \cite{JCao2005,MGupta1994,CLi2010}. To apply our theory to this problem
we denote by $|W|$ the matrix with entries $|W|_{ij}=|W_{ij}|$. Using this allows us to prove the following general stability condition for this class of Cohen-Grossberg neural networks.

\begin{theorem}\label{CGtheorem} \textbf{(Stability of Cohen-Grossberg Neural Networks)}
Let $(\mathcal{F},\mathbb{R}^n)$ be the Cohen-Grossberg network given by (\ref{CG}) where $\phi_i$ has Lipschitz constant $\mathcal{L}$. If
$|1-\epsilon|+\mathcal{L}\rho(|W|)<1$ then $(\mathcal{F},\mathbb{R}^n)$ has a globally attracting fixed point.
\end{theorem}

\begin{proof}
Assuming $\mathcal{F}$ is given by (\ref{CG}) then $\mathcal{F}\in C^1_{\infty}(\mathbb{R})$. The claim then is that for the local systems and interaction given by (\ref{CGloc}) and (\ref{CGint}) the matrix $A=\tilde{\Lambda}^T\cdot diag[L_1,\dots,L_n]$ with
\begin{align*}
&\max_{\textbf{x}\in X}|\varphi^\prime_{i}(x_i)|=L_i=|1-\epsilon|, \ \ \text{and}\\
&\max_{\textbf{x}\in X}|(DF)_{ji}(\textbf{x})|\leq\tilde{\Lambda}_{ij}=
\begin{cases}
1+\left|\frac{W_{ji}\mathcal{L}}{1-\epsilon}\right| \ \ &\text{for} \ \ i=j\\
\left|\frac{W_{ji}\mathcal{L}}{1-\epsilon}\right| \ \ &\text{for} \ \ i\neq j
\end{cases}
\end{align*}
is a stability matrix of $(\mathcal{F},X)$. To see this note that the constants $$\Lambda_{ij}=\max_{\textbf{x}\in X}|(DF)_{ji}(\textbf{x})|$$
satisfy (\ref{eq2.3}). Since $\Lambda_{ij}\leq\tilde{\Lambda}_{ij}$ then similarly the constants $\tilde{\Lambda}_{ij}$ satisfy (\ref{eq2.3}) verifying the claim.

Notice that the matrix $A$ has the form
\begin{equation}\label{M1}
A=\left[
\begin{array}{cccc}
|1-\epsilon|+|W_{11}\mathcal{L}| & |W_{12}\mathcal{L}| & \dots &  |W_{1n}\mathcal{L}|\\
|W_{21}\mathcal{L}| & |1-\epsilon|+|W_{22}\mathcal{L}| &  \dots & |W_{2n}\mathcal{L}|\\
\vdots & \vdots & \ddots & \vdots\\
|W_{n1}\mathcal{L}| & |W_{n2}\mathcal{L}| & \dots & |1-\epsilon|+|W_{nn}\mathcal{L}|\\
\end{array}
\right].
\end{equation}
The spectral radius of the matrix $A$ is then
$$\rho(A)=\rho\big(|1-\epsilon|I_n+\mathcal{L}|W|\big)=|1-\epsilon|+\mathcal{L}\cdot\rho(|W|)$$
where $I_n$ is the $n\times n$ identity matrix. From theorem \ref{stability} it follows that $(\mathcal{F},\mathbb{R}^n)$ has a globally attracting fixed point if $|1-\epsilon|+\mathcal{L}\rho\big(|W|\big)<1$ implying the result.
\end{proof}

To the best of our knowledge theorem \ref{CGtheorem} is a new stability criteria for this class of neural networks. However, the major goal of this paper is to extend such results to the case where the network's interactions include time delays. To do so requires a better understanding of the \emph{graph structure} of dynamical networks.

\subsection{Graph Structure of Dynamical Networks}
To each dynamical network $(\mathcal{F},X)$ there is an associated unweighted directed graph called its \textit{graph of interactions}. An \textit{unweighted directed graph} $G$ is an ordered pair $G=(V,E)$ where the sets $V$ and $E$ are the \textit{vertex set} and \textit{edge set} of $G$ respectively. If the vertex set $V=\{v_1,\dots,v_n\}$ then we denote the directed edge from $v_i$ to $v_j$ by $e_{ij}$.

\begin{definition}
The graph $\Gamma_{\mathcal{F}}=(V,E)$ with vertex set $V=\{v_1,\dots,v_n\}$ and edge set $E=\{e_{ij}: i\in \mathcal{I}_j, \ j\in\mathcal{I}\}$ is called the \textit{graph of interactions} of the dynamical network $(\mathcal{F},X)$.
\end{definition}

Each vertex $v_i\in V$ of $\Gamma_{\mathcal{F}}=(V,E)$ corresponds to the $i$th element of the dynamical network $(\mathcal{F},X)$. Also, there is an edge $e_{ij}\in E$ or a \emph{directed interaction} from the $i$th to the $j$th element of the network if and only if the $j$th component of the interaction $F(\textbf{x})$ depends on the $i$th coordinate of $\textbf{x}$.

\begin{figure}
  \begin{center}
    \begin{overpic}[scale=.5]{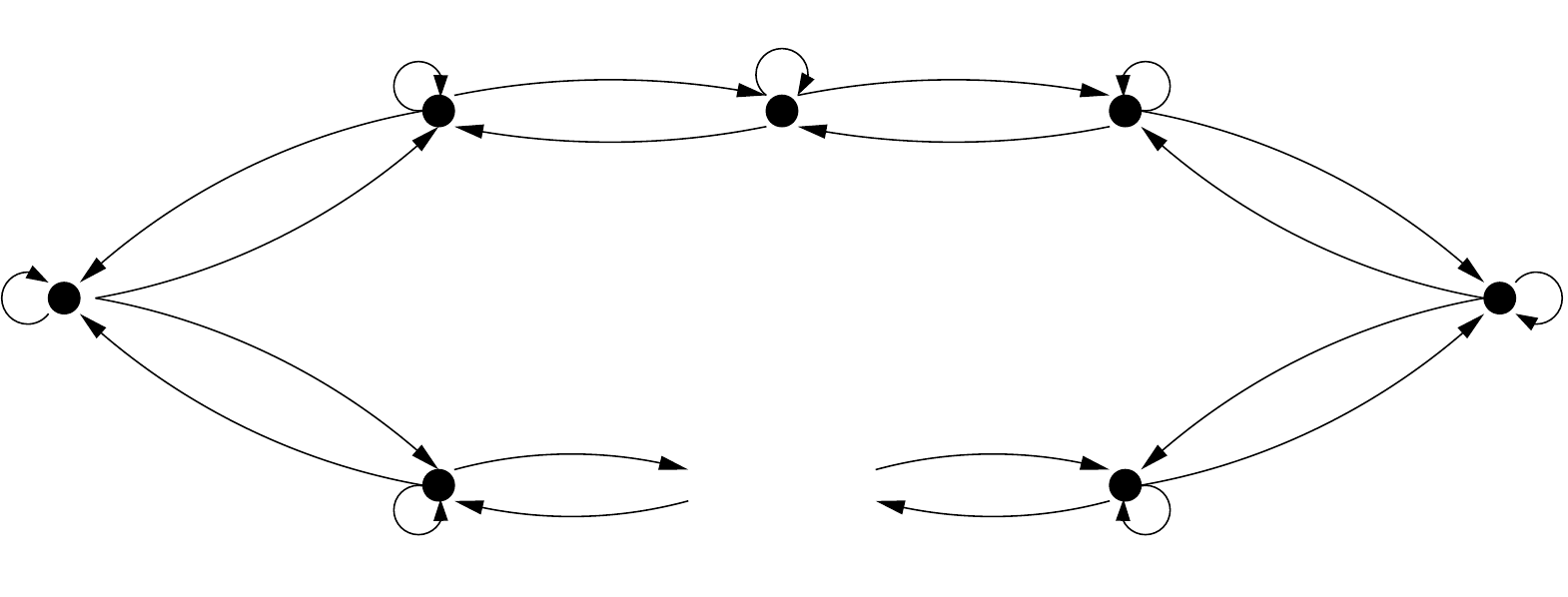}
    \put(49,-5){$\Gamma_\mathcal{F}$}
    \put(26,37){$v_1$}
    \put(-7,17){$v_2$}
    \put(26,-1.5){$v_3$}
    \put(66,-1.5){$v_{n-3}$}
    \put(102,17){$v_{n-2}$}
    \put(68,37){$v_{n-1}$}
    \put(48,37){$v_{n}$}
    \put(48,6.5){$\dots$}
    \end{overpic}
  \end{center}
  \caption{The graph of interactions of the Cohen-Grossberg type network in example 1.}\label{fig1}
\end{figure}
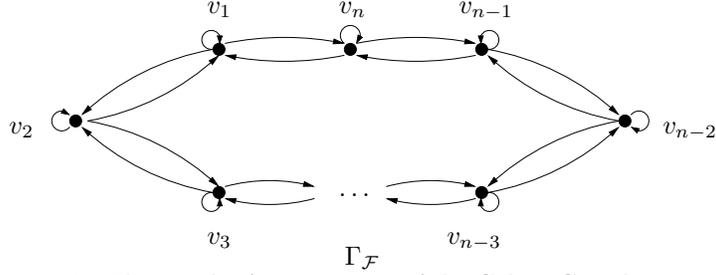

\begin{example}
Suppose the map $\mathcal{F}:X\rightarrow X$ is given by
$$\mathcal{F}_j(\mathbf{x})=(1-\epsilon)x_j+a\big[\tanh(b x_{j-1})+\tanh(b x_{j+1})\big]+c_j$$
for $1\leq j\leq n$ where the indices are taken mod $n$, $X_i=\mathbb{R}$, $a,b,c_j\in\mathbb{R}$, and $\epsilon\neq 1$. Notice that this has the form of a Cohen-Grossberg network given by (\ref{CG}) where $\phi_i(x_i)=\tanh(b x_i)$ has Lipschitz constant $|b|$. The graph of interaction of the system $(\mathcal{F},\mathcal{R}^n)$ is shown in figure 1 for $\epsilon\neq 1$.

Moreover, it follows that the matrix $|W|$ is given by
$$|W|=\left[
\begin{array}{ccccccc}
0 & |a| & &  &|a|\\
|a| & 0 & |a| & & \\
 & \ddots & \ddots & \ddots & \\
 & & |a| & 0 & |a|\\
|a| &  &  & |a| & 0\\
\end{array}
\right]$$
for this system. Since $|W|$ has nonnegative entries and constant row sums $2|a|$ then $\rho(|W|)=2|a|$. By theorem \ref{CGtheorem}, if $|1-\epsilon|+2|ab|<1$ then the dynamical network $(\mathcal{F},\mathbb{R}^n)$ has a globally attracting fixed point.
\end{example}

\subsection{Dynamical Networks Without Local Dynamics}
By definition \ref{dynamicalnetwork} the function $\mathcal{F}=F\circ \varphi$ is the composition of the network's local dynamics $\varphi$ and interaction $F$. However, if the system has \emph{no local dynamics}, i.e. $\varphi=id$ is the identity map, then the dynamical network $(\mathcal{F},X)=(F,X)$ is simply the interaction $\mathcal{F}=F$ on $X$.

Conversely, notice that $\mathcal{F}=F\circ \varphi$ can itself be considered to be an interaction. Writing $\mathcal{F}=(F\circ \varphi)\circ id$ the dynamical network $(\mathcal{F},X)$ is simply the interaction $\mathcal{F}=F\circ\varphi$ with no local dynamics. Since this is formally the same dynamical network we call it the dynamical network $(\mathcal{F},X)$ \emph{considered as a network without local dynamics}.

\begin{proposition}\label{prop-1}
Suppose $(\mathcal{F},X)$ is the dynamical network with no local dynamics. If the constants $\Lambda_{ij}$ satisfy (\ref{eq2.3}) for $\mathcal{F}=F$ then $\Lambda^T$ is a stability matrix of $(\mathcal{F},X)$.
\end{proposition}

\begin{proof}
If $(\mathcal{F},X)$ has no local dynamics then $\varphi_i=id$ implying $L_i=1$ satisfy equation (\ref{eq0.1}). Hence, the matrix $\Lambda^T\cdot diag[L_1,\dots, L_n]=\Lambda^T$ is a stability matrix of $(\mathcal{F},X)$ under the assumption that the constants $\Lambda_{ij}$ satisfy (\ref{eq2.3}) for $\mathcal{F}=F$.
\end{proof}

\begin{remark}
If $\mathcal{F}\in C^1_{\infty}(X)$ then proposition \ref{prop-1} implies that the matrix $\Lambda^T$ with
$$\Lambda_{ij}=\max_{\textbf{x}\in X}|D\mathcal{F}_{ji}(\textbf{x})|$$
is a stability matrix of $(\mathcal{F},X)$ if the network has no local dynamics.
\end{remark}

\begin{proposition}\label{prop-2}
If $A$ is a stability matrix of $(\mathcal{F},X)$ then $A$ is a stability matrix of $(\mathcal{F},X)$ considered as a network without local dynamics.
\end{proposition}

\begin{proof}
Suppose $A$ is a stability matrix of $(\mathcal{F},X)$ where $\mathcal{F}=F\circ \varphi$. Then $A=\Lambda^T\cdot diag[L_1,...,L_n]$ implying $S_{ij}=\Lambda_{ji}L_j$ where $L_i$ and $\Lambda_{ij}$ satisfy (\ref{eq0.1}) and (\ref{eq2.3}) respectively. Therefore,
\begin{align*}
d\big(\mathcal{F}_j(\textbf{x}),\mathcal{F}_j(\textbf{y})\big)= d\big(F_j(\varphi(\textbf{x})|_{\mathcal{I}_j}),F_j(\varphi(\textbf{x})|_{\mathcal{I}_j})\big)&\leq\\
\sum_{i\in\mathcal{I}_j}\Lambda_{ij}d\big(\varphi_i(x_i),\varphi_i(y_i)\big)\leq \sum_{i\in\mathcal{I}_j}\Lambda_{ij}L_i d(x_i,y_i)&
\end{align*}
for all $j\in\mathcal{I}$ and $\textbf{x},\textbf{y}\in X$. Since $A^T_{ij}=\Lambda_{ij}L_i$ then proposition \ref{prop-1} implies that  $A$ is a stability matrix of $(\mathcal{F},X)$ considered as a network without local dynamics.
\end{proof}

The reason we consider dynamical networks without local dynamics is that this notion is useful in the study of the global stability of time-delayed networks introduced in the following section.

\section{Global Stability of Time-Delayed Dynamical Networks}
Our first task in this section is to extend the formalism of the previous section to dynamical networks with time-delays. This requires that we allow the state $\mathcal{F}^{k+1}(\textbf{x})$ of the dynamical network $(\mathcal{F},X)$ to depend not only on $\mathcal{F}^k(\textbf{x})$ but on some subset of the previous $T$ states $\{\mathcal{F}^{k}(\textbf{x}),\mathcal{F}^{k-1}(\textbf{x}),\dots,\mathcal{F}^{k-(T-1)}(\textbf{x})\}$ of the system.

\subsection{Time-Delayed Dynamical Networks}
Given the fixed integer $T\geq 1$ let $\mathcal{T}=\{0,-1,\dots,-T+1\}$ and let $X^\tau=X$ for all $\tau\in\mathcal{T}$. We define the product space
$$X^T=\bigoplus_{\tau\in\mathcal{T}}X^{\tau}.$$
For the local systems $(\varphi,X)$ we define the function $\varphi:X^T\rightarrow X^T$ as follows. For $(\textbf{x}^0,\dots,\textbf{x}^{-T+1})\in X^T$ let
$$\varphi(\textbf{x}^0,\dots,\textbf{x}^{-T+1})=\big(\varphi(\textbf{x}^0),\dots,\varphi(\textbf{x}^{-T+1})\big)$$
where $\varphi(\textbf{x}^0,\dots,\textbf{x}^{-T+1})^\tau_i=\varphi_i(x_i^\tau)$ for $(i,\tau)\in\mathcal{I}\times\mathcal{T}$. Adopting the terminology of the previous section we say $(\varphi,X^T)$ are local systems.

\begin{definition}\label{DDN}
A map $H:X^T\rightarrow X$ is called a \textit{time-delayed interaction} if for all $j\in\mathcal{I}$ there is a set $\mathcal{I}^j\subseteq\mathcal{I}\times\mathcal{T}$ and a continuous function
$$H_j:\bigoplus_{(i,\tau)\in \mathcal{I}^j} X_i^\tau\rightarrow X_j.$$
The map $H$ is defined as follows:
$$H(\mathbf{x})_j=H_j(\textbf{x}|_{\mathcal{I}^j}), \ \ j\in \mathcal{I}, \ \ \text{and} \ \ \textbf{x}\in X^T.$$
The superposition $\mathcal{H}=H\circ \varphi$ generates the \textit{time-delayed dynamical network} $(\mathcal{H},X^T)$. The \textit{orbit} of $(\mathbf{x}^0,\mathbf{x}^{-1},\dots,\mathbf{x}^{-T+1})\in X^T$ under $\mathcal{H}$ is the sequence $\{\mathbf{x}^k\}_{k>-T}$ where
$$\mathbf{x}^{k+1}=\mathcal{H}(\mathbf{x}^{k},\mathbf{x}^{k-1},\dots,\mathbf{x}^{k-(T-1)}).$$
\end{definition}

As before, we are concerned with finding sufficient conditions under which a time-delayed dynamical network has a globally attracting fixed point.

\begin{definition}
A \textit{fixed point} of a delayed dynamical network $(\mathcal{H},X^T)$ is an $\tilde{\textbf{x}}\in X$ such that $\tilde{\mathbf{x}}=\mathcal{H}(\tilde{\mathbf{x}},\dots,\tilde{\mathbf{x}})$. The fixed point $\tilde{\textbf{x}}\in X$ is a \emph{global attractor} of $(\mathcal{H},X^T)$ if for any initial condition $(\mathbf{x}^0,\mathbf{x}^{-1},\dots,\mathbf{x}^{-T+1})\in X^T$ the limit
$$\lim_{k\rightarrow\infty}d_{max}(\mathbf{x}^k,\tilde{\mathbf{x}})=0.$$
\end{definition}

To connect the study of delayed dynamical networks with the dynamical networks introduced in section 2 we construct the following. For the delayed dynamical system $(\mathcal{H},X^T)$ we define the index set
$$\mathcal{I}_\mathcal{H}=\{i,j;\ell,m : \ 1\leq j \leq n, \ (i,m)\in\mathcal{I}^j, \ 1\leq \ell\leq m\}.$$

Let $\tilde{\mathcal{H}}_j:X^T\rightarrow X_j$ be the function $\mathcal{H}_j(\mathbf{x}^{k},\dots,\mathbf{x}^{k-(T-1)})$ in which each \emph{time-delayed variable} $x_i^{k-m}$ is replaced by $x_{i,j;m,m}$ for $m\geq 1$ and each $x_i^k$ by $x_i$. For each $\delta\in\mathcal{I}_\mathcal{H}$ let $\mathcal{H}_{\delta}=\mathcal{H}_{i,j;\ell,m}$ be the function
\begin{equation}\label{eq5}
\mathcal{H}_{i,j;\ell,m}(x_{i,j;\ell-1,m})=x_{i,j;\ell-1,m}
\end{equation}
where $x_{i,j;0,m}=x_i$. Define the product space
\begin{equation}\label{eq3}
X_\mathcal{H}=\Big\{\textbf{x}\in X\oplus\Big(\bigoplus_{\delta\in\mathcal{I}_\mathcal{H}}X_\delta\Big):x_{i,j;\ell,m}=x_{s,t;u,v} \ \text{if} \ i=s, \ \ell=u\Big\}
\end{equation}
where $X_{i,j;\ell,m}$ is a copy of the space $X_i$. Let $\mathcal{N}_\mathcal{H}:X_\mathcal{H}\rightarrow X_\mathcal{H}$ be the function
$$\mathcal{N}_\mathcal{H}=\Big(\bigoplus_{j\in\mathcal{I}}\tilde{\mathcal{H}}_j\Big) \oplus\Big(\bigoplus_{\delta\in\mathcal{I}_\mathcal{H}}\mathcal{H}_{\delta}\Big).$$

The function $\mathcal{N}_\mathcal{H}:X_\mathcal{H}\rightarrow X_\mathcal{H}$ generates the dynamical network $(\mathcal{N}_\mathcal{H},X_\mathcal{H})$. However, because of the time-delays involved and the potential noninvertability of the local systems $(\varphi,X^T)$, the dynamical network $(\mathcal{N}_\mathcal{H},X_{\mathcal{H}})$ cannot typically be represented as a network with local dynamics. We therefore consider $(\mathcal{N}_\mathcal{H},X_{\mathcal{H}})$ as a network with no local dynamics.

\begin{theorem}\label{next}
The time-delayed dynamical network $(\mathcal{H},X^T)$ has a globally attracting fixed point if and only if the same is true of $(\mathcal{N}_{\mathcal{H}},X_{\mathcal{H}})$.
\end{theorem}

\begin{proof}
To compare the dynamical systems $(\mathcal{H},X^T)$ and $(\mathcal{N}_{\mathcal{H}},X_{\mathcal{H}})$ we define the mapping $\underline{\mathcal{H}}:X^T\rightarrow X^T$ by
$$(\mathbf{x}^{k+1},\mathbf{x}^{k},\dots, \mathbf{x}^{k-T+2})=\underline{\mathcal{H}}(\mathbf{x}^{k},\mathbf{x}^{k-1},\dots,\mathbf{x}^{k-(T-1)})$$
for all $k\geq 0$. The claim is that the systems $(\underline{\mathcal{H}},X^T)$ and $(\mathcal{N}_{\mathcal{H}},X_{\mathcal{H}})$ are conjugate.

To verify this we let the map $\pi:X^T\rightarrow X_{\mathcal{H}}$ be given by
$$\pi(\textbf{y})_\delta=\begin{cases}
y_i^0, &\delta=i\in\mathcal{I}\\
y^{-\ell}_i, &\delta=i,j;\ell,m\in\mathcal{I}_{\mathcal{H}}
\end{cases}.$$
We note that it follows from (\ref{eq3}) that $\pi$ is a bijection.

Let $\textbf{x}=(\textbf{x}^0,\dots,\textbf{x}^{-T+1})\in X^T$. As $\pi\circ\underline{\mathcal{H}}(\textbf{x})=\pi\big(\mathcal{H}(\textbf{x}),\textbf{x}^0,\dots, \textbf{x}^{-T+2}\big)$ then
$$\big(\pi\circ\underline{\mathcal{H}}(\textbf{x})\big)_\delta=\begin{cases}
\mathcal{H}_i(\textbf{x}), &\delta=i\in\mathcal{I}\\
x^{-\ell+1}_i, &\delta=i,j;\ell,m\in\mathcal{I}_{\mathcal{H}}
\end{cases}.$$
Also, we have
$$\big(\mathcal{N}_\mathcal{H}\circ\pi(\textbf{x})\big)_\delta=\begin{cases}
\tilde{\mathcal{H}}_i\big(\pi(\textbf{x})\big), &\delta=i\in\mathcal{I}\\
\mathcal{H}_{i,j;\ell,m}\big(\pi(\textbf{x})_{i,j;\ell-1,m}\big), &\delta=i,j;\ell,m\in\mathcal{I}_{\mathcal{H}}
\end{cases}.$$

For $i\in\mathcal{I}$, $\tilde{\mathcal{H}}_i:X^T\rightarrow X_i$ is the function $\mathcal{H}_i$ in which each variable $x_i^{-m}$ is replaced by $x_{i,j;m,m}$ and $x_i^0$ by $x_i$. As this is the same as mapping $\textbf{x}$ to $\pi(\textbf{x})$ then $\mathcal{H}_i(\textbf{x})=\tilde{\mathcal{H}}_i(\pi\big(\textbf{x})\big)$. Moreover, as $\mathcal{H}_{i,j;\ell,m}\big(\pi(\textbf{x})_{i,j;\ell-1,m}\big)=\pi(\textbf{x})_{i,j;\ell-1,m}=x_i^{-\ell+1}$ then
$$\pi\circ\underline{\mathcal{H}}(\textbf{x})=\mathcal{N}_{\mathcal{H}}\circ\pi(\textbf{x}).$$
As $\pi$ is a bijection then $(\underline{\mathcal{H}},X^T)$ and $(\mathcal{N}_{\mathcal{H}},X_{\mathcal{H}})$ are conjugate verifying the claim. Therefore, $(\underline{\mathcal{H}},X^T)$ has a globally attracting fixed point if and only if $(\mathcal{N}_{\mathcal{H}},X_{\mathcal{H}})$ has a globally attracting fixed point.

Since $\mathcal{H}(\textbf{x})$ is simply the restriction $\underline{\mathcal{H}}(\textbf{x})|_{X^0}$ then $(\underline{\mathcal{H}},X^T)$ is globally stable if and only if the same is true of $(\mathcal{H},X^T)$ implying the result.
\end{proof}

By combining the results of theorems \ref{stability} and \ref{next} it is possible to investigate the dynamic stability of $(\mathcal{H},X^T)$ via its associated dynamical network $(\mathcal{N}_{\mathcal{H}},X_{\mathcal{H}})$.

\begin{corollary}\label{corollary1}
Suppose $A$ is a stability matrix of $(\mathcal{N}_{\mathcal{H}},X_{\mathcal{H}})$. If $\rho(A)<1$ then $(\mathcal{H},X^T)$ has a globally attracting fixed point.
\end{corollary}

In light of corollary \ref{corollary1} we say $A$ is a stability matrix of $(\mathcal{H},X^T)$ if it is a stability matrix of $(\mathcal{N}_{\mathcal{H}},X_{\mathcal{H}})$. Moreover, if $\mathcal{N}_{\mathcal{H}}\in C_\infty^1(X_{\mathcal{H}})$ we will write $\rho(\mathcal{H})=\rho(\mathcal{N}_{\mathcal{H}})$.

\begin{example}
Consider the delayed dynamical network $(\mathcal{H},X^T)$ given by
\begin{equation}\label{eq2.1}
\mathcal{H}(\textbf{x}^{k},\textbf{x}^{k-1},\textbf{x}^{k-2},\textbf{x}^{k-3})=\left[\begin{array}{l}
(1-\epsilon)x^{k-1}_1+2a\tanh(b x^{k-3}_{2})+c_1\\
(1-\epsilon)x^{k-1}_2+2a\tanh(b x^{k-3}_{1})+c_2
\end{array}
\right]
\end{equation}
with local systems $\varphi_i(x_i)=(1-\epsilon)x_i+c_i$ for $i=1,2$ and interaction
\begin{equation}\label{eq2.2}
H(\textbf{x}^{k},\textbf{x}^{k-1},\textbf{x}^{k-2},\textbf{x}^{k-3})=\left[\begin{array}{l}
x^{k-1}_1+2a\tanh(b \frac{x^{k-3}_{2}-c_2}{1-\epsilon})\\
x^{k-1}_2+2a\tanh(b \frac{x^{k-3}_{1}-c_1}{1-\epsilon})
\end{array}
\right]
\end{equation}
where $a,b,c_i\in\mathbb{R}$, $X=\mathbb{R}^2$, and $\epsilon\neq1$. Note that this system is a delayed version of the Cohen-Grossberg network described in example 1 for $n=2$. Such time-delayed systems have been of recent interest. See for instance \cite{LTao2011,SChena2009,LWang2005}.

Since $\mathcal{I}^1=\{(1,-1),(2,-3)\}$, $\mathcal{I}^2=\{(1,-3),(2,-1)\}$ and $T=4$ then
$$\mathcal{I}_{\mathcal{H}}=\{11;11, \ 12;13, \ 12;23, \ 12;33, \ 21;13, \ 21;23, \ 21;33, \ 22;11\}$$
and the components of the function $\tilde{\mathcal{H}}$ are given by
$$\tilde{\mathcal{H}}=\left[\begin{array}{l}
(1-\epsilon)x_{11;11}+2a\tanh(b x_{21;33})+c_1\\
(1-\epsilon)x_{22;11}+2a\tanh(b x_{12;33})+c_2
\end{array}
\right].$$
Following our construction the function $\mathcal{N}_{\mathcal{H}}$ is then defined as
$$\mathcal{N}_{\mathcal{H}}(\textbf{x})=\left[\begin{array}{l}
(\mathcal{N}_{\mathcal{H}})_1(x_{11;11},x_{21;33})\\
(\mathcal{N}_{\mathcal{H}})_2(x_{12;33},x_{22;11})\\
(\mathcal{N}_{\mathcal{H}})_{11;11}(x_{11;01})\\
(\mathcal{N}_{\mathcal{H}})_{12;13}(x_{12;03})\\
(\mathcal{N}_{\mathcal{H}})_{12;23}(x_{12;13})\\
(\mathcal{N}_{\mathcal{H}})_{12;33}(x_{12;23})\\
(\mathcal{N}_{\mathcal{H}})_{21;13}(x_{21;03})\\
(\mathcal{N}_{\mathcal{H}})_{21;23}(x_{21;13})\\
(\mathcal{N}_{\mathcal{H}})_{21;33}(x_{21;23})\\
(\mathcal{N}_{\mathcal{H}})_{22;11}(x_{22;01})
\end{array}
\right]=
\left[\begin{array}{l}
(1-\epsilon)x_{11;11}+2a\tanh(b x_{21;33})+c_1\\
(1-\epsilon)x_{22;11}+2a\tanh(b x_{12;33})+c_2\\
x_1\\
x_1\\
x_{12;13}\\
x_{12;23}\\
x_2\\
x_{21;13}\\
x_{21;23}\\
x_2
\end{array}
\right].$$
Here, $\textbf{x}\in X_{\mathcal{H}}=\mathbb{R}^{10}$ where the component $x_{11;11}=x_{12;13}$ and the component $x_{21;13}=x_{22;11}$. As $\mathcal{N}_{\mathcal{H}}\in C_\infty^1(\mathbb{R}^{10})$ then proposition \ref{prop-1} implies that the constants $$\Lambda_{ij}=\max_{\textbf{x}\in \mathbb{R}^{10}}|(D\mathcal{N}_{\mathcal{H}})_{ji}(\textbf{x})|$$ form the stability matrix
$$\Lambda^T=\left[
\begin{array}{cccccccccc}
0 & 0 & |1-\epsilon| & 0 & 0 & 0 & 0 & 0 & 2|ab| & 0\\
0 & 0 & 0 & 0 & 0 & 2|ab| & 0 & 0 & 0 & |1-\epsilon|\\
1 & 0 & 0 & 0 & 0 & 0 & 0 & 0 & 0 & 0\\
1 & 0 & 0 & 0 & 0 & 0 & 0 & 0 & 0 & 0\\
0 & 0 & 0 & 1 & 0 & 0 & 0 & 0 & 0 & 0\\
0 & 0 & 0 & 0 & 1 & 0 & 0 & 0 & 0 & 0\\
0 & 1 & 0 & 0 & 0 & 0 & 0 & 0 & 0 & 0\\
0 & 0 & 0 & 0 & 0 & 0 & 1 & 0 & 0 & 0\\
0 & 0 & 0 & 0 & 0 & 0 & 0 & 1 & 0 & 0\\
0 & 1 & 0 & 0 & 0 & 1 & 0 & 0 & 0 & 0\\
\end{array}
\right].$$
From this one can compute that
\begin{equation}\label{eq3.5}
\rho(\mathcal{H})=\sqrt{\frac{|1-\epsilon|+\sqrt{|1-\epsilon|^2+8|ab|}}{2}}.
\end{equation}

Importantly, $\rho(\mathcal{H})<1$ if and only if $|1-\epsilon|+2|ab|<1$. By corollary \ref{corollary1} the time-delayed interaction $H$ stabilizes the local systems $(\varphi,X^T)$ if $|1-\epsilon|+2|ab|<1$, which is the same condition derived in example 1 for the global stability of the undelayed version of this system.
\end{example}

With this example in mind we turn our attention to determining how modifications of an interaction's time-delays effects whether the associated time-delayed network has a globally attracting fixed point.

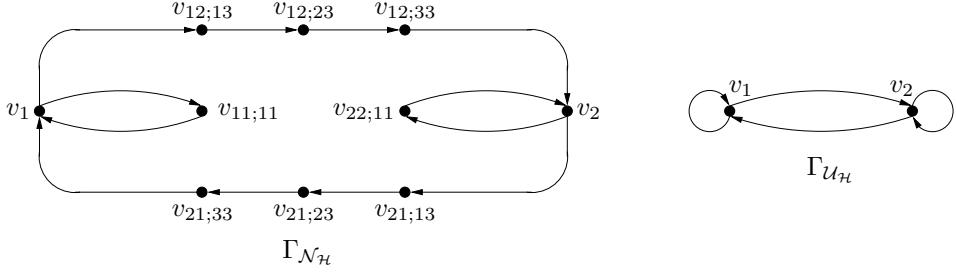
\begin{figure}
  \begin{center}
    \begin{overpic}[scale=.425]{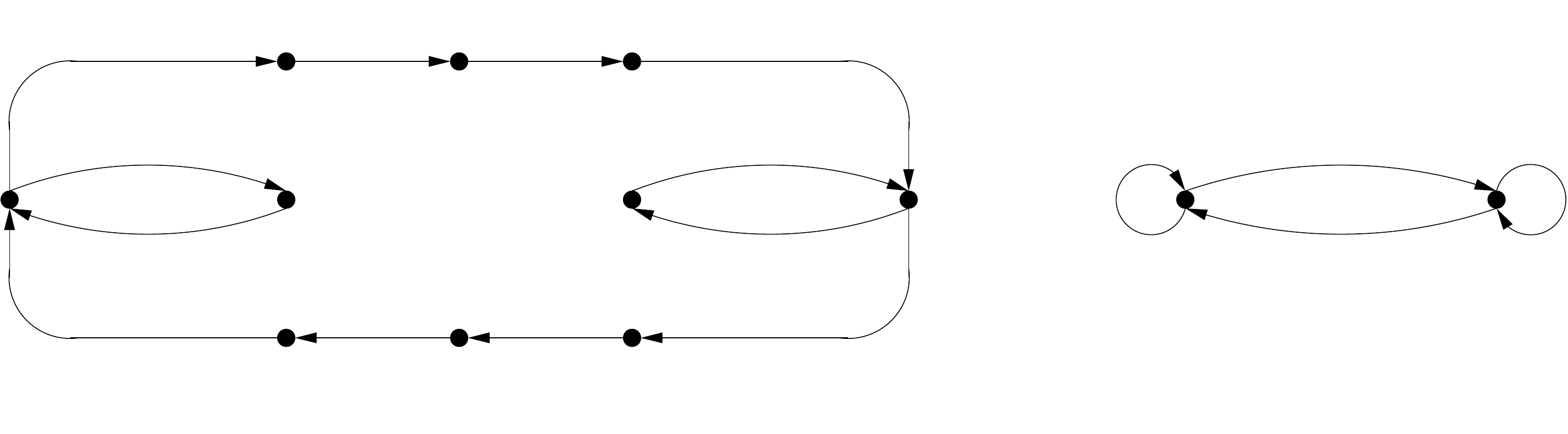}
    \put(27,-1){$\Gamma_{\mathcal{N}_{\mathcal{H}}}$}
    \put(-3,14.5){$v_1$}
    \put(59,14.5){$v_2$}
    \put(20,14.5){$v_{11;11}$}
    \put(32.5,14.5){$v_{22;11}$}
    \put(15,25.5){$v_{12;13}$}
    \put(26,25.5){$v_{12;23}$}
    \put(37,25.5){$v_{12;33}$}
    \put(15,3.5){$v_{21;33}$}
    \put(26,3.5){$v_{21;23}$}
    \put(37,3.5){$v_{21;13}$}
    \put(75.5,17){$v_1$}
    \put(93,17){$v_2$}
    \put(84,8){$\Gamma_{\mathcal{U}_{\mathcal{H}}}$}
    \end{overpic}
  \end{center}
  \caption{The graph of interactions $\Gamma_{\mathcal{N}_{\mathcal{H}}}$ and $\Gamma_{\mathcal{U}_{\mathcal{H}}}$ corresponding to $(\mathcal{N}_{\mathcal{H}},X_{\mathcal{H}})$ and $(\mathcal{U}_{\mathcal{H}},X)$ in examples 2 and 3 respectively.}\label{fig1}
\end{figure}

\section{Removing Time-Delays}
In this section we describe how modifying an interaction's time-delays effects a network's stability. Our first result is that if a time-delayed dynamical network is known to have stable dynamics then removing these time-delays will not destabilize the network. However, the converse of this statement does not hold. In fact, we give an example of an undelayed network with stable dynamics that becomes unstable with the addition of delays.

On the other hand, if time-delays are introduced in a specific way into a network's interaction the time-delays will not be able to destabilize the network. Interactions with this property are called \emph{non-distributed}. This concept allows us to study the dynamic stability of delayed networks with such interactions in terms of undelayed networks, which is significantly simpler from a computational point of view.

\subsection{Undelayed Dynamical Networks}
Here we formalize the notion of modifying a system's time-delays. Our first objective is to describe how to change a delayed dynamical network to a network without time-delays.
\begin{definition}
Let $(\mathcal{H},X^T)$ be the time-delayed dynamical network with local systems $(\varphi,X^T)$ and interaction $H:X^T\rightarrow X$. The map $\mathcal{U}_{\mathcal{H}}:X\rightarrow X$ given by $$\mathcal{U}_{\mathcal{H}}(\textbf{x})=\mathcal{H}(\textbf{x},\dots,\textbf{x})$$ generates the \emph{undelayed dynamical network} $(\mathcal{U}_{\mathcal{H}},X)$ with local systems $(\varphi,X)$ and interaction $\mathcal{U}_H:X\rightarrow X$ defined as
$$\mathcal{U}_{H}(\textbf{x})=H(\textbf{x},\dots,\textbf{x}).$$
\end{definition}

Simply put, $(\mathcal{U}_{\mathcal{H}},X)$ is the delayed dynamical network $(\mathcal{H},X^T)$ with its delays removed. The following result relates the dynamic stability of the undelayed network $(\mathcal{U}_{\mathcal{H}},X)$ to the stability of $(\mathcal{H},X^T)$.

\begin{theorem}\label{theorem3}
Suppose $A$ is a stability matrix of the time-delayed dynamical network $(\mathcal{H},X^T)$. If $\rho(A)<1$ then the undelayed dynamical network $(\mathcal{U}_{\mathcal{H}},X)$ has a globally attracting fixed point.
\end{theorem}

That is, if a time-delayed interaction is known to stabilize or maintain the stability of a set of local systems then the undelayed version of this interaction will do the same (see example 4). However, the converse of theorem \ref{theorem3} does not hold as the following example shows.

\begin{example} \textbf{(A Stable Undelayed Network Associated with an Unstable Delayed Network)}
Consider the time-delayed dynamical network $(\mathcal{H},X^T)$ given by
$$\mathcal{H}(\textbf{x}^{k},\textbf{x}^{k-1})=
\left[\begin{array}{l}
(1-\epsilon)x_1^k+\tanh(bx^k_2)-\tanh(bx^{k-1}_2)\\
(1-\epsilon)x_2^k+\tanh(bx^k_1)-\tanh(bx^{k-1}_1)
\end{array}\right]$$
where $\epsilon=1/2$, $b=1$, and $X=\mathbb{R}^2$. For simplicity we consider $(\mathcal{H},X^T)$ as a delayed network with no local dynamics. Notice that the map
$$\mathcal{N}_\mathcal{H}(\textbf{x})=
\left[\begin{array}{l}
(\mathcal{N}_\mathcal{H})_1(x_1,x_2,x_{21;11})\\
(\mathcal{N}_\mathcal{H})_2(x_1,x_2,x_{12;11})\\
(\mathcal{N}_\mathcal{H})_{12;11}(x_1)\\
(\mathcal{N}_\mathcal{H})_{21;11}(x_2)
\end{array}\right]=
\left[\begin{array}{l}
(1-\epsilon)x_1+\tanh(bx_2)-\tanh(bx_{21;11})\\
(1-\epsilon)x_2+\tanh(bx_1)-\tanh(bx_{12;11})\\
x_1\\
x_2
\end{array}\right]
$$
has the fixed point $\tilde{\textbf{x}}=(0,0,0,0)\in\mathbb{R}^4$. As the matrix
$$D\mathcal{N}_\mathcal{H}(\tilde{\textbf{x}})=
\left[\begin{array}{cccc}
1-\epsilon&b&0&-b\\
b&1-\epsilon&-b&0\\
1&0&0&0\\
0&1&0&0
\end{array}\right]
$$ it follows that $\rho\big(D\mathcal{N}_\mathcal{H}(\tilde{\textbf{x}})\big)=(1+\sqrt{17})/4>1$. Hence, $\tilde{\textbf{x}}$ is a repelling fixed point of $(\mathcal{N}_{\mathcal{H}},X_{\mathcal{H}})$.

Since $(\mathcal{N}_{\mathcal{H}},X_{\mathcal{H}})$ does not have a globally attracting fixed point the same holds for $(\mathcal{H},X^T)$ by theorem \ref{next}. However, the undelayed dynamical network $(\mathcal{U}_{\mathcal{H}},X)$ is given by
$$\mathcal{U}_{\mathcal{H}}(\textbf{x})=
\left[\begin{array}{l}
(1-\epsilon)x_1\\
(1-\epsilon)x_2
\end{array}\right]=
\frac{1}{2}\left[\begin{array}{l}
x_1\\
x_2
\end{array}\right]
$$ which has $(0,0)\in\mathbb{R}^2$ as a globally attracting fixed point.

Hence, it is possible when removing time-delays to effect the global stability of a system. Moreover, even though $(\mathcal{U}_{\mathcal{H}},X)$ has a globally attracting fixed point the network $(\mathcal{H},X^T)$ has no stability matrix $A$ with the property $\rho(A)<1$.
\end{example}

Theorem \ref{theorem3} does however have a partial converse if we restrict ourselves to specific types of time-delayed interactions. Recall from definition \ref{DDN} that a time-delayed interaction is a function $H:X^T\rightarrow X$ given by $$H_j:\bigoplus_{(i,\tau)\in \mathcal{I}^j} X_i^\tau\rightarrow X_j$$
where the set $\mathcal{I}^j\subseteq\mathcal{I}\times\mathcal{T}$.

\begin{definition}\label{SD}
The time-delayed interaction $H:X^T\rightarrow X$ is called \emph{non-distributed} if for all $j\in\mathcal{I}$ the set $\mathcal{I}^j\subseteq\mathcal{I}\times\mathcal{T}$ has the property that
$$(i,\mu)\in\mathcal{I}^j \Rightarrow (i,\nu)\notin\mathcal{I}^j \ \ \text{for all} \ \ \nu\neq \mu.$$
If a time-delayed interaction does not have this property we say it is \emph{distributed}.
\end{definition}

Equivalently, $H:X^T\rightarrow X$ is a \emph{non-distributed interaction} if for each $i,j\in\mathcal{I}$ the component $H_j$ depends on at most one variable of the form $x_i^\tau$. If $(\mathcal{H},X^T)$ has an interaction which is non-distributed then we say the network $(\mathcal{H},X^T)$ is also non-distributed and write $(\mathcal{H},X^T)\in nd(X^T)$.

\begin{theorem}\label{theorem3.45}
Suppose $(\mathcal{H},X^T)\in nd(X^T)$. Then there is a stability matrix $A$ of $(\mathcal{H},X^T)$ with the property $\rho(A)<1$ if and only if there is a stability matrix $\tilde{A}$ of $(\mathcal{U}_{\mathcal{H}},X)$ with $\rho(\tilde{A})<1$.
\end{theorem}

Therefore, if a time-delayed dynamical network is stable and has a non-distributed interaction then any change in the interaction's time-delays that maintains this property does not effect the network's stability. The following is an immediate corollary of theorems \ref{stability} and \ref{theorem3.45}.

\begin{corollary}
Suppose $(\mathcal{H},X^T)\in nd(X^T)$ and that $A$ and $\tilde{A}$ are stability matrices of $(\mathcal{H},X^T)$ and $(\mathcal{U}_{\mathcal{H}},X)$ respectively. If either $\rho(A)<1$ or $\rho(\tilde{A})<1$ then both $(\mathcal{H},X^T)$ and $(\mathcal{U}_{\mathcal{H}},X)$ are globally stable.
\end{corollary}

\begin{example}
Consider the time-delayed dynamical network $(\mathcal{H},X^T)$ given by (\ref{eq2.1}) in example 2. The undelayed version of $(\mathcal{H},X^T)$ is the dynamical network $(\mathcal{U}_{\mathcal{H}},X)$ given by
$$\mathcal{U}_{\mathcal{H}}(\textbf{x})=\left[\begin{array}{l}
(1-\epsilon)x_1+2a\tanh(b x_2)+c_1\\
(1-\epsilon)x_2+2a\tanh(b x_1)+c_2
\end{array}
\right]$$
with local systems $\varphi_i(x_i)=(1-\epsilon)x_i+c_i$ and interaction
\begin{equation*}
\mathcal{U}_H(\textbf{x})=\left[\begin{array}{l}
x_1+2a\tanh(b \frac{x_{2}-c_2}{1-\epsilon})\\
x_2+2a\tanh(b \frac{x_{1}-c_1}{1-\epsilon})
\end{array}
\right]
\end{equation*}
where $i\in\{1,2\}$, $a,b,c_i\in\mathbb{R}$, $X=\mathbb{R}^2$, and $\epsilon\neq1$. Using  $\tilde{\Lambda}_{ij}=\max_{\textbf{x}\in X}|(D\mathcal{U}_H)_{ji}(\textbf{x})|$ and the constants $L_i=\max_{x\in X_i}|\varphi^\prime_i(x_i)|$ one has the stability matrix
$$\tilde{\Lambda}\cdot diag[L_1,L_2]=\left[
\begin{array}{cc}
|1-\epsilon|&2|ab|\\
2|ab|&|1-\epsilon|
\end{array}
\right]$$
of $(\mathcal{U}_{\mathcal{H}},\mathbb{R}^2)$ from which it follows that
\begin{equation}\label{eq5.1}
\rho(\mathcal{U}_{\mathcal{H}})=|1-\epsilon|+2|ab|.
\end{equation}
As can be seen from (\ref{eq2.2}) the dynamical network $(\mathcal{H},X^T)\in nd(X^T)$. Theorem \ref{theorem3.45} then implies that the time-delayed interaction $H$ stabilizes the local systems $(\varphi,X^T)$ if $|1-\epsilon|+2|ab|<1$.

It is important to note that the same condition was formulated in example 2. However, the computations used to find $\rho(\mathcal{U}_{\mathcal{H}})$ in this example are much simpler than those required to find $\rho(\mathcal{H})$ in example 2. In fact, from a computational point of view, it is always easier to analyze the stability of an undelayed dynamical network $(\mathcal{U}_{\mathcal{H}},X)$ than the time-delayed network $(\mathcal{H},X^T)$.
\end{example}

We note that if a time-delayed dynamical network $(\mathcal{H},X^T)$ is distributed, i.e. has a distributed time-delayed interaction, then the conclusions of theorem \ref{theorem3.45} may not hold. For example, the dynamical network $(\mathcal{U}_{\mathcal{H}},X^T)$ in example 3 has the stability matrix
$$\tilde{A}=
\left[\begin{array}{cc}
1/2&0\\
0&1/2
\end{array}\right]$$
with $\rho(\tilde{A})<1.$ However, $(\mathcal{H},X^T)$ does not have a globally attracting fixed point. The reason being is that $(\mathcal{H},X^T)\notin nd(X^T)$.

Example 3 also points out that certain types of time delays can be destabilizing to a system. On the other hand, theorem \ref{theorem3.45} states that if time-delays are introduced into a stable system in specific ways the resulting system will remain stable. More precisely, if delays are introduced into an undelayed interaction $F$ in a way that results in a non-distributed interaction $H$ this modification will not destabilize the network's dynamics.

\subsection{Cohen-Grossberg Networks with Time-Delays}
To apply the results of section 4.1 to the Cohen-Grossberg networks considered in this paper suppose $(\mathcal{H},X^T)$ is the time-delayed dynamical network given by
\begin{equation}\label{eq23}
\mathcal{H}_j(\textbf{x}^{k},\dots,\textbf{x}^{k-(T-1)})=
(1-\epsilon)x^{k-\tau_{jj}}_{j}+\sum_{i=1}^n W_{ij}\phi(x_i^{k-\tau_{ij}})+c_j
\end{equation}
with local systems $\varphi_i(x_i)=(1-\epsilon)x_i+c_i$ and time-delayed interaction
\begin{equation}\label{eq24}
H(\textbf{x}^{k},\dots,\textbf{x}^{k-(T-1)})=
x^{k-\tau_{jj}}_j+\sum_{i=1}^n W_{ij}\phi\Big(\frac{x_i^{k-\tau_{ij}}-c_i}{1-\epsilon}\Big).
\end{equation}
Here, $1\leq i,j\leq n$, $0\leq \tau_{ij}\leq T-1$, and $\epsilon\neq1$. As before $W\in\mathbb{R}^{n\times n}$, $X=\mathbb{R}^n$, and $\phi:\mathbb{R}\rightarrow\mathbb{R}$ is a smooth function with Lipschitz constant $\mathcal{L}$.

Equations (\ref{eq23}) and (\ref{eq24}) describes the Cohen-Grossberg network $(\mathcal{H},\mathbb{R}^{nT})$ with time-delays. Since $(\mathcal{H},X^T)\in nd(X^T)$ then as a corollary to theorem \ref{CGtheorem} and theorem \ref{theorem3.45} we have the following result.

\begin{theorem}\label{CGdel} \textbf{(Stability of Time-Delayed Cohen-Grossberg Networks)}
Let $(\mathcal{H},\mathbb{R}^{nT})$ be the time-delayed Cohen-Grossberg network given by (\ref{eq23}) where $\phi_i$ has Lipschitz constant $\mathcal{L}$. If $|1-\epsilon|+\mathcal{L}\rho(|W|)<1$ then $(\mathcal{H},\mathbb{R}^{nT})$ has a globally attracting fixed point.
\end{theorem}

The point of theorem \ref{CGdel} is that although time-delayed Cohen-Grossberg networks are more complicated systems than  Cohen-Grossberg networks without delays, the criteria for their stability is not.

\subsection{Proving Theorem \ref{theorem3} and Theorem \ref{theorem3.45}}
In order to prove theorems \ref{theorem3} and theorem \ref{theorem3.45} we will use the well known theorem of Perron and Frobenius and the following standard terminology.

A \emph{weighted directed graph} $G=(V,E,\omega)$ is a graph with vertices $V$, edges $E$, where the edge $e_{ij}\in E$ has weight $\omega(e_{ij})$. A \textit{path} $P$ in $G$ is an ordered sequence of distinct vertices $v_1,\dots,v_m\in V$ such that $e_{i,i+1}\in E$ for $1\leq i\leq m-1$. The vertices $v_2,\dots,v_{m-1}$ are the \emph{interior vertices} of $P$. In the case that the interior vertices are distinct, but $v_1=v_m$, then $P$ is a \textit{cycle}. A cycle $v_1,\dots,v_m$ is called a \emph{loop} if $m=1$.

A directed graph $G$ is called \emph{strongly connected} if there is a path from each vertex in the graph to every other vertex or $G$ has a single vertex. The \emph{strongly connected components} of $G$ are its maximal strongly connected subgraphs.

The \emph{weighted adjacency matrix} of a weighted directed graph $G=(V,E,\omega)$ is the matrix
$M$ with entries $M_{ij}=\omega(e_{ij})$ where $M_{ij}=0$ if $e_{ij}\notin E$. We say the graph $G$ is the graph \emph{associated} with its weighted adjacency matrix $M$. If $M\in\mathbb{R}^{n\times n}$ then $M$ is said to be \textit{irreducible} if the graph $G$ associated with $M$ is strongly connected. Moreover, the matrix $M$ is \textit{nonnegative} if $M_{ij}\geq 0$ for all $1\leq i,j\leq n$.

\begin{theorem}\textbf{(Perron-Frobenius)}\label{PF}
Let $M\in\mathbb{R}^{n\times n}$ and suppose that $M$ is irreducible and nonnegative. Then

\indent (a) $\rho(M)>0$;\\
\indent (b) $\rho(M)$ is an eigenvalue of $M$;\\
\indent (c) $\rho(M)$ is an algebraically simple eigenvalue of $M$; and\\
\indent (d) the left and right eigenvectors $\textbf{x}$ and $\textbf{y}$ associated with $\rho(M)$ have strictly positive entries.
\end{theorem}

If a graph $G$ is not strongly connected then it has strongly connected components $\mathbb{S}(G)_1\dots,\mathbb{S}(G)_N$. Let $M$ be the matrix associated with $G$ and $M_j$ the matrix associated with $\mathbb{S}_{j}(G)$. Then the eigenvalues are
\begin{equation}\label{eq1000}
\sigma(M)=\bigcup_{j=1}^N\sigma\big(M_j\big)
\end{equation}
from which it follows that
\begin{equation}\label{eq1001}
\rho(M)=\max_{1\leq j\leq N}\rho(M_j).
\end{equation}
We call a strongly connected component $\mathbb{S}_j(G)$ \textit{trivial} if it consists of a single vertex without loop in which case $\sigma\big(\mathbb{S}_j(G)\big)=\{0\}$.
The following result is found in \cite{Horn85}.

\begin{proposition}\label{proposition2}
Let $A,B\in\mathbb{R}^{n\times n}$ be nonnegative and suppose that A is irreducible. Then $\rho(A + B) > \rho(A)$ if $B\neq 0$.
\end{proposition}

That is, the spectral radius of a nonnegative irreducible matrix is strictly monotonic in each of its entries.

Suppose $M\in\mathbb{R}^{n\times n}$ is nonnegative and $\alpha>0,\mathcal{L}\geq 0$ such that $\alpha+\mathcal{L}=M_{\ell m}$ for fixed integers $1\leq\ell,m\leq n$. Let $M_\theta=M_{\theta}(\alpha,\mathcal{L})\in\mathbb{R}^{n+1\times n+1}$ be the modified matrix $M$ given by
$$(M_\theta)_{ij}=\begin{cases}
M_{ij} \ &\text{for} \ \ 1\leq i,j\leq n, \ (i,j)\neq (\ell,m)\\
\alpha \ &\text{for} \ \ (i,j)=(0,m)\\
\mathcal{L} \ &\text{for} \ \ (i,j)=(\ell,m)\\
\theta \ &\text{for} \ \ (i,j)=(\ell,0)\\
0 &\text{otherwise}
\end{cases}
$$
for $0\leq i,j\leq n$ and $\theta>0$. If $G$ and $G_\theta$ are the graphs associated with $M$ and $M_\theta$ respectively the difference between these graphs is shown in figure 3.

\begin{figure}
  \begin{center}
    \begin{overpic}[scale=.425]{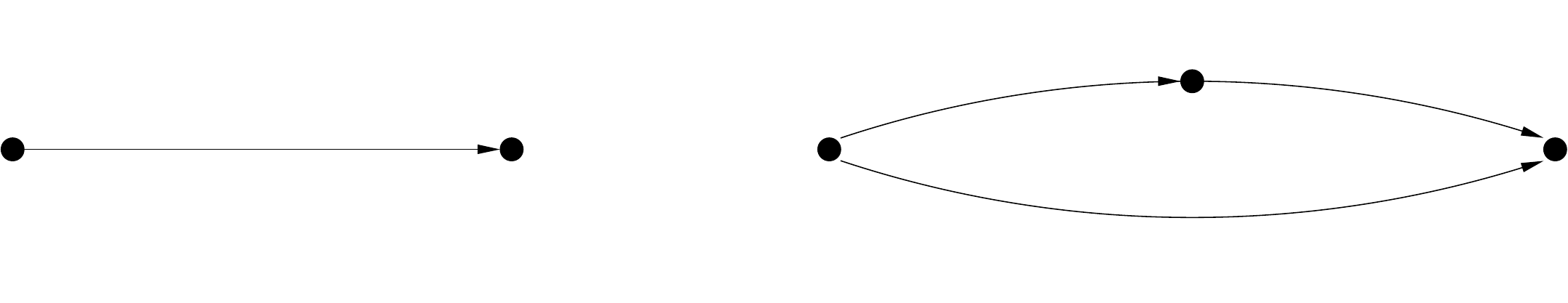}
    \put(-4,9){$v_\ell$}
    \put(34,9){$v_m$}
    \put(13,0){$G$}
    \put(12,12){$M_{\ell m}$}
    \put(48,9){$v_\ell$}
    \put(101,9){$v_m$}
    \put(74.5,16){$v_0$}
    \put(86,14){$\alpha$}
    \put(65,14){$\theta$}
    \put(75,7){$\mathcal{L}$}
    \put(74.5,0){$G_\theta$}
    \end{overpic}
  \end{center}
  \caption{The edge $e_{\ell m}$ in $G$ and its modified form in $G_\theta$ where $M_{\ell m}=\alpha+\mathcal{L}$}\label{fig30}
\end{figure}
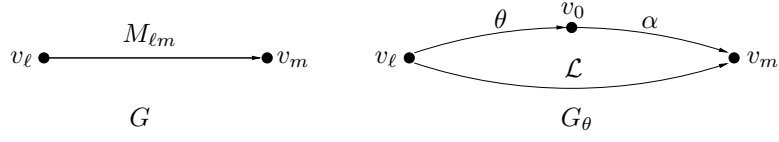

\begin{lemma}\label{lemma103}
Suppose $M\in\mathbb{R}^{n\times n}$ is nonnegative and irreducible. Then\\
\indent (1) $\theta\leq \rho(M_\theta)\leq\rho(M)$ if and only if $\theta\leq\rho(M)$; and\\
\indent (2) $\rho(M)<\rho(M_\theta)<\theta$ if and only if $\theta>\rho(M)$.
\end{lemma}

\begin{proof}
Suppose $M\in\mathbb{R}^{n\times n}$ is nonnegative and irreducible. For $\theta>0$ and a nonzero $\lambda\in\mathbb{R}$, let $M_\theta-\lambda I_{n+1}$ be the block matrix
$$M_\theta-\lambda I_{n+1}=\left[
\begin{array}{cc}
A & B\\
C & D
\end{array}
\right]$$
where $A\in\mathbb{R}^{1\times 1}$ and $I_{n+1}\in\mathbb{R}^{n+1 \times n+1}$ is the identity matrix. Then the matrix $A=[\lambda]$ and $A^{-1}=[1/\lambda]$.
\begin{equation}\label{identity1}
\text{Using the identity} \ \det\left(\left[\begin{array}{cc}
A & B\\
C & D
\end{array}\right]\right)=\det(A)\cdot\det(D-CA^{-1}B)
\end{equation}
it follows that $\det(M_\theta-\lambda I_{n+1})=-\lambda\det(\tilde{M}_\theta(\lambda)-\lambda I_n)$ where $\tilde{M}_\theta(\lambda)\in\mathbb{R}^{n\times n}$ is the matrix given by
$$\big(\tilde{M}_\theta(\lambda)\big)_{ij}=\begin{cases}
M_{ij} \ &\text{for} \ \ (i,j)\neq (\ell,m)\\
\alpha\frac{\theta}{\lambda}+\mathcal{L} \ &\text{for} \ \ (i,j)=(\ell,m)
\end{cases}
$$
for $1\leq i,j\leq n$. Therefore, $\tilde{M}_\theta(\theta)=M$ as $\alpha+\mathcal{L}=M_{\ell m}$.

As $M$ is nonnegative and irreducible then both $\rho(M)>0$ and $\rho(M)\in\sigma(M)$ by the Perron-Frobenius theorem. Hence,
$$\det\big(M_{\rho(M)}-\rho(M)I_{n+1}\big)=-\rho(M)\det\big(\tilde{M}_{\rho(M)}(\rho(M))-\rho(M) I_n\big)=0.$$ Therefore, $\rho(M)\in \rho(M_{\rho(M)})$.

If $G$ and $G_0$ are the graphs associated with $M$ and $M_{\rho(M)}(0,M_{\ell m})$ respectively then $G$ and $G_0$ have the same nontrivial strongly connected components. Hence,
$$\rho\big(M_{\rho(M)}(0,M_{\ell m})\big)=\rho(M).$$
Since the eigenvalues of a matrix are continuous with respect to the matrix entries then parts (b) and (c) of the Perron-Frobenius theorem imply $\rho(M)=\rho\big(M_{\rho(M)}\big)$ so long as $\alpha>0,\mathcal{L}\geq 0$ and $\alpha+\mathcal{L}=M_{\ell m}$.

Let $\gamma>1$. As $M_{\rho(M)}$ is assumed to be both nonnegative and irreducible then proposition \ref{proposition2} implies $$\rho(M_{\rho(M)})<\rho\big(M_{\gamma\rho(M)}\big)<\rho\big(\gamma M_{\rho(M)}\big).$$
Since $\rho\big(\gamma M_{\rho(M)}\big)=\gamma\rho(M)$ then letting $\theta=\gamma\rho(M)$ implies
$$\rho(M)<\rho\big(M_{\theta}\big)<\theta \ \ \text{if and only if} \ \ \theta>\rho(M)$$
as $\theta>0$. Similarly, for $0\leq \gamma\leq 1$ it follows that
$$\theta\leq \rho(M_\theta)\leq\rho(M) \ \text{if and only if} \ \theta\leq\rho(M)$$
completing the proof.
\end{proof}

If the matrix $M$ in lemma \ref{lemma103} happens to be irreducible then the following holds.

\begin{corollary}\label{corollary3}
Suppose $M\in\mathbb{R}^{n\times n}$ is nonnegative. Then $\rho(M_\theta)<\theta$ if and only if $\rho(M)<\theta$.
\end{corollary}

\begin{proof}
Suppose $M\in\mathbb{R}^{n\times n}$ is nonnegative. If $M$ is irreducible then the result holds by lemma \ref{lemma103}. If $M$ is reducible then its associated graph $G$ is not strongly connected and can be decomposed into the strongly connected components $$\mathbb{S}(G)_1,\dots,\mathbb{S}(G)_N.$$ Let $M_i$ be the matrix associated with the strongly connected component $\mathbb{S}(G)_i$ and notice that each $M_i$ is nonnegative.

If $v_\ell,v_m\in\mathbb{S}(G)_i$ we consider two cases. First, suppose $\mathbb{S}(G)_i$ is nontrivial. Then the graph $G_\theta$ associated with $M_\theta$ has the strongly connected components
$$\mathbb{S}(G)_1\dots,\mathbb{S}_\theta(G)_i,\dots, \mathbb{S}(G)_N$$
where $(M_\theta)_i$ is the nonnegative matrix associated with $\mathbb{S}_\theta(G)_i$. As the matrix associated with a nontrivial strongly connected component is irreducible it follows from lemma \ref{lemma103} that $\rho\big((M_\theta)_i\big)<\theta$ if and only if $\rho(M)<\theta$. From equation (\ref{eq1001}),
$$\rho(M_\theta)=\max\{\rho(M_1),\dots,\rho\big((M_\theta)_i\big),\dots,\rho(M_N)\}= \max\{\rho(M),\rho\big((M_\theta)_i\big)\}.$$
Hence, if $\rho(M)<\theta$ then $\rho\big((M_\theta)_i\big)<\theta$ implying $\rho(M_\theta)<\theta$. Conversely, if $\rho(M_\theta)<\theta$ then $\rho(M)<\theta$ and the result holds in this case.

If $\mathbb{S}(G)_i$ is trivial then $\alpha,\mathcal{L}=0$ and $M_\theta$ has the strongly connected components
$$\mathbb{S}(G)_1,\dots,\mathbb{S}(G)_N,\mathbb{S}(G)_{N+1}$$
where $\mathbb{S}(G)_{N+1}$ is trivial with single vertex $v_0$. Hence, $\sigma(M_{N+1})=\{0\}$ implying $\rho(M_\theta)=\rho(M)$ by equation (\ref{eq1000}). The result then follows in the case where the vertices $v_\ell,v_m\in\mathbb{S}(G)_i$.

If $v_\ell\in\mathbb{S}(G)_i$ and $v_m\in\mathbb{S}(G)_j$ for $i\neq j$ then the graphs associated with $M$ and $M_\theta$ have the same nontrivial strongly connected components. Equation (\ref{eq1001}) again implies that $\rho(M_\theta)=\rho(M)$ which yields the result.
\end{proof}

A proof of theorem \ref{theorem3} is the following.

\begin{proof}
Consider the time-delayed dynamical network $(\mathcal{H},X^T)$ with stability matrix $\Lambda^T$. For $\tilde{\mathcal{I}}=\mathcal{I}\cup\mathcal{I}_{\mathcal{H}}$ corollary \ref{prop-1} and equation (\ref{eq2.3}) imply
\begin{equation}\label{eq22}
d(\mathcal{N}_{\mathcal{H}}(\textbf{x}|_{\tilde{\mathcal{I}}})_j,\mathcal{N}_{\mathcal{H}}(\textbf{y}|_{\tilde{\mathcal{I}}})_j)\leq \sum_{i\in\tilde{\mathcal{I}}}\Lambda_{ij}d(x_i,y_i)
\end{equation}
for all $j\in\tilde{\mathcal{I}}$ and $\textbf{x},\textbf{y}\in X_{\mathcal{H}}$ as $(\mathcal{N}_{\mathcal{H}},X_{\mathcal{H}})$ has no local dynamics. For $\delta=ij;\ell m\in\mathcal{I}_{\mathcal{H}}$ let $\delta-1=ij;\ell-1,m$. Since
\begin{equation}\label{eq5.5}
\mathcal{H}_{\delta}(x_{\delta-1})=x_{\delta-1} \ \ \text{for all} \ \ \delta\in\mathcal{I}_{\mathcal{H}}
\end{equation}
then without loss in generality we may assume
\begin{equation}\label{eq9}
\Lambda_{j,\delta}=
\begin{cases}
1 & \text{if} \ \ j=\delta-1\\
0 & \text{otherwise}
\end{cases} \ \ \text{for} \ \ j\in\mathcal{I}_\mathcal{H}.
\end{equation}
as these constants are the smallest to satisfy equation (\ref{eq2.3}). We now consider two cases.\\

\emph{Case 1:} Suppose $x_\ell^{k-1}$ is a variable of $\mathcal{H}_m$ for $\ell,m\in\mathcal{I}$. Let $\bar{\mathcal{H}}_m$ be the function $\mathcal{H}_m$ in which the time-delayed variable $x_\ell^{k-1}$ is replaced by $x_\ell^{k-0}$ and define
\begin{equation}\label{eq8}
\bar{\mathcal{H}}=\Big(\bigoplus_{
\begin{smallmatrix}
j\in\mathcal{I},\\
j\neq m
\end{smallmatrix}}\mathcal{H}_j\Big)\oplus(\bar{\mathcal{H}}_m).
\end{equation}
Note that $\big(\mathcal{N}_{\bar{\mathcal{H}}}\big)_m$ is then the function $\big(\mathcal{N}_{\mathcal{H}}\big)_m$ in which the variable $x_{\ell m;11}$ is replaced by $x_\ell$. Let $\mathcal{I}^*=\tilde{\mathcal{I}}-\{\ell m;11\}$. From equation (\ref{eq22}) it follows that
\begin{equation*}
d(\mathcal{N}_{\mathcal{H}}(\textbf{x}|_{\tilde{\mathcal{I}}})_m,\mathcal{N}_{\mathcal{H}} (\textbf{y}|_{\tilde{\mathcal{I}}})_m)\leq \sum_{i\in\mathcal{I^*}-\{\ell\}}\Lambda_{im}d(x_i,y_i)+\Lambda_{\ell m}d(x_\ell,y_\ell)+\Lambda_{\mu,m}d(x_\mu,y_{\mu})
\end{equation*} for all $\textbf{x},\textbf{y}\in X_{\mathcal{H}}$ where $\mu=\ell m;11$.
Hence,
$$d(\mathcal{N}_{\bar{\mathcal{H}}}(\textbf{x}|_{\mathcal{I}^*})_m, \mathcal{N}_{\bar{\mathcal{H}}}(\textbf{y}|_{\mathcal{I}^*})_m)\leq \sum_{i\in\mathcal{I}^*-\{\ell\}}\Lambda_{im}d(x_i,y_i)+\big(
\Lambda_{\ell m}+\Lambda_{\mu,m}\big)d(x_\ell,y_\ell)$$
for all $\textbf{x},\textbf{y}\in X_{\bar{\mathcal{H}}}$. As $\bar{\mathcal{H}}_j=\mathcal{H}_j$ for all $j\neq m$ then the matrix $\tilde{\Lambda}^T$ given by
\begin{equation}\label{eq10}
\tilde{\Lambda}_{ij}=
\begin{cases}
\Lambda_{ij} & (i,j)\neq (\ell,m)\\
\Lambda_{\ell m}+\Lambda_{\mu,m} & (i,j)=(\ell,m)
\end{cases} \ \ \text{for} \ \ i,j\in\mathcal{I}^*
\end{equation}
is a stability matrix of $(\mathcal{N}_{\bar{\mathcal{H}}},X_{\bar{\mathcal{H}}})$.

\begin{figure}
  \begin{center}
    \begin{overpic}[scale=.425]{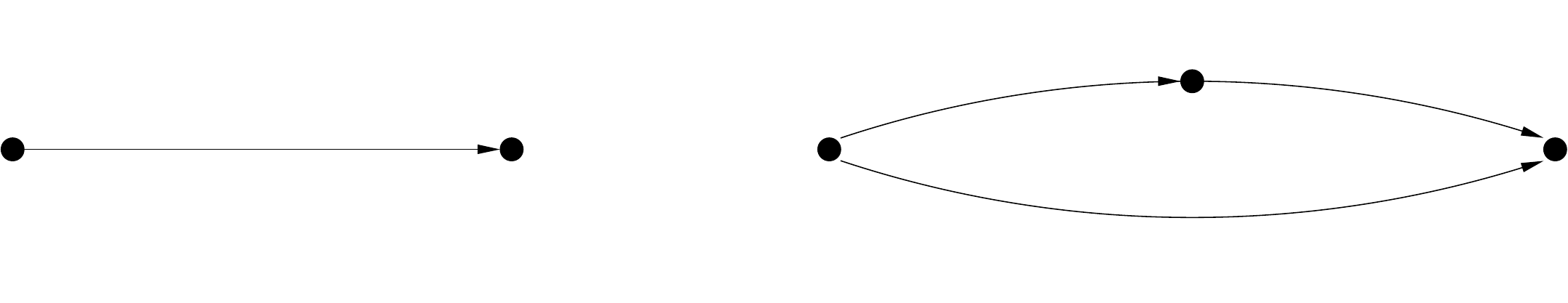}
    \put(-4,9){$v_\ell$}
    \put(34,9){$v_m$}
    \put(15,0){$\tilde{G}$}
    \put(7,12){$\Lambda_{\ell m}+\Lambda_{\mu,m}$}
    \put(48,9){$v_\ell$}
    \put(101,9){$v_m$}
    \put(74.5,16){$v_0$}
    \put(87,14){$\Lambda_{\mu,m}$}
    \put(64,14){$1$}
    \put(73,7){$\Lambda_{\ell m}$}
    \put(69,0){$\tilde{G}_1=G$}
    \end{overpic}
  \end{center}
  \caption{The edge $e_{\ell m}$ in $\tilde{G}$ and its modified form in $\tilde{G}_1=G$ corresponding to case 1 in the proof of theorem \ref{theorem3}.}\label{fig300}
\end{figure}

By identifying the index $\mu$ with 0, equations (\ref{eq9}) and (\ref{eq10}) imply that the matrix $\tilde{\Lambda}_\theta(\alpha,\mathcal{L})=\Lambda$ for $\theta=1$, $\alpha=\Lambda_{\mu,m}$, and $\mathcal{L}=\Lambda_{\ell m}$. Since $\tilde{\Lambda}$ is nonnegative then corollary \ref{corollary3} implies $\rho(\tilde{\Lambda})<1$ if $\rho(\Lambda)<1$. Hence, $\rho(\tilde{\Lambda}^T)<1$ if $\rho(\Lambda^T)<1$.\\

\emph{Case 2:} For $\tau>1$ and $q,m\in\mathcal{I}$ suppose $x_q^{k-\tau}$ is a variable of $\mathcal{H}_m$ and $x_q^{k-(\tau-1)}$ is not. Let $\bar{\mathcal{H}}_m$ be the function $\mathcal{H}_m$ in which the time-delayed variable $x_q^{k-\tau}$ is replaced by $x_q^{k-(\tau-1)}$. Similar to case 1, let $\bar{\mathcal{H}}$ be defined as in (\ref{eq8}) where the indices of the dynamical network $(\mathcal{N}_{\bar{\mathcal{H}}},X_{\bar{\mathcal{H}}})$ are relabeled such that $qm;i-1,\tau-1$ is now $qm;i\tau$ for all $1<i\leq\tau$.

Let $\mathcal{I}_*=\tilde{\mathcal{I}}-\{qm;1\tau\}$ and the index $qm;2\tau=\nu$. Then
$$(\mathcal{N}_{\bar{\mathcal{H}}})_j=
\begin{cases}
(\mathcal{N}_{\mathcal{H}})_j & j\neq \nu\\
x_{q} & j=\nu
\end{cases} \ \ \text{for} \ \ j\in\mathcal{I}_*
$$
implying $\displaystyle{d(\mathcal{N}_{\bar{\mathcal{H}}}(\textbf{x}|_{\mathcal{I}_*})_j, \mathcal{N}_{\bar{\mathcal{H}}}(\textbf{y}|_{\mathcal{I}_*})_j)=d(x_q,y_q) \ \ \text{for} \ \ j=\nu}$.
This together with equation (\ref{eq22}) implies that the matrix $\tilde{\Lambda}^T$ given by
\begin{equation}\label{eq11}
\tilde{\Lambda}_{ij}=
\begin{cases}
\Lambda_{ij} & (i,j)\neq (q,\nu)\\
1 & (i,j)=(q,\nu)
\end{cases} \ \ \text{for} \ \ i,j\in\mathcal{I}_*
\end{equation}
is a stability matrix of $(\mathcal{N}_{\bar{\mathcal{H}}},X_{\bar{\mathcal{H}}})$.

Let $\eta=qm,\tau\tau$. By identifying the index $\eta$ with 0 and $\eta-1$ with $\ell$, equations (\ref{eq9}) and (\ref{eq11}) imply $\tilde{\Lambda}_\theta(\alpha,\mathcal{L})=\Lambda$ for $\theta=1$, $\alpha=\Lambda_{\eta-1,\eta}$, and $\mathcal{L}=0$. As in case 1, it follows that
$\rho(\tilde{\Lambda}^T)<1$ if $\rho(\Lambda^T)<1$.\\

Observe that the dynamical network $(\mathcal{U}_{\mathcal{H}},X)$ can be obtained by sequentially replacing each timed-delayed variable $x_\ell^{k-\tau}$ of $(\mathcal{H},X^T)$ by $x_\ell^{k-(\tau-1)}$ as in case 1 and case 2. Specifically, this is done by first replacing all variables of the form $x_\ell^{k-1}$ then $x_\ell^{k-2}$ and so on. Therefore, if $\rho(\Lambda^T)<1$ then there exists a stability matrix $\tilde{A}$ of $(\mathcal{U}_\mathcal{H},X)$ with $\rho(\tilde{A})<1$ where  $\tilde{A}$ is a stability matrix of $(\mathcal{U}_\mathcal{H},X)$ considered as a network without local dynamics. Theorem \ref{stability} therefore implies that $(\mathcal{U}_{\mathcal{H}},X)$ has a globally attracting fixed point.
\end{proof}

We note that the proof of theorem \ref{theorem3} gives a step by step procedure for constructing a stability matrix of $(\mathcal{U}_{\mathcal{H}},X_{\mathcal{H}})$ from that of $(\mathcal{H},X^T)$. This process is depicted in figures \ref{fig300} and \ref{fig301} which corresponds to cases 1 and 2 in the proof of theorem \ref{theorem3}. The graphs $\tilde{G}$ and $G$ correspond to the matrices $\tilde{\Lambda}$ and $\Lambda$ respectively.

\begin{figure}
  \begin{center}
    \begin{overpic}[scale=.425]{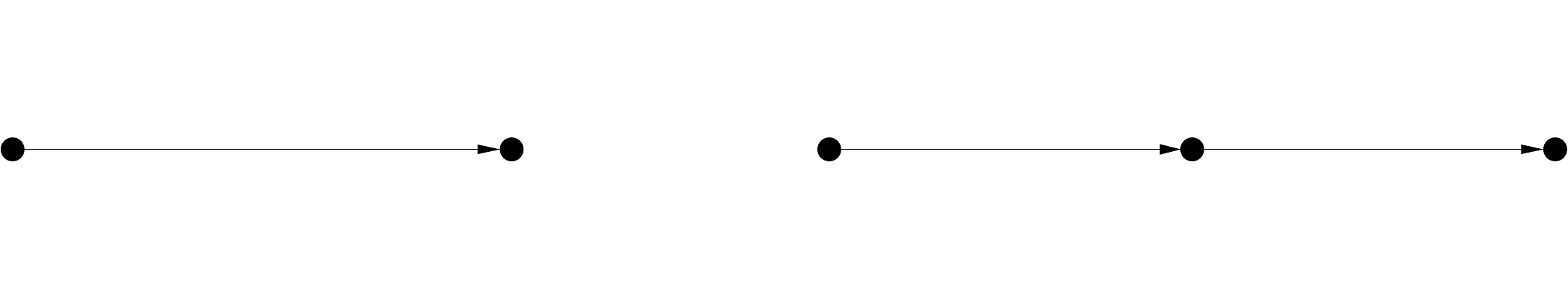}
    \put(-4,9){$v_\ell$}
    \put(34,9){$v_m$}
    \put(15,1.5){$\tilde{G}$}
    \put(11,12){$\Lambda_{\nu-1,\nu}$}
    \put(48,9){$v_\ell$}
    \put(101,9){$v_m$}
    \put(74.5,12){$v_0$}
    \put(83,12){$\Lambda_{\eta-1,\eta}$}
    \put(63,12){$1$}
    \put(69,1.5){$\tilde{G}_1=G$}
    \end{overpic}
  \end{center}
  \caption{The edge $e_{\ell m}$ in $\tilde{G}$ and its modified form in $\tilde{G}_1=G$ corresponding to case 2 in the proof of theorem \ref{theorem3}.}\label{fig301}
\end{figure}
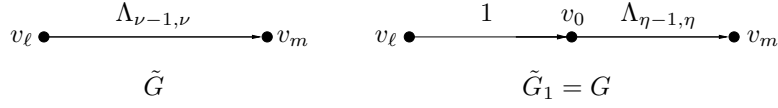

Moreover, with respect to case 1 of theorem \ref{theorem3} we note the following. Suppose $\tilde{\Lambda}^T$ is a stability matrix of $(\mathcal{N}_{\bar{\mathcal{H}}},X_{\bar{\mathcal{H}}})$ and that the constants $$\Lambda_{ij}=\tilde{\Lambda}_{ij} \ \
\text{for} \ \  (i,j)\neq(\ell,m) \ \  \text{where} \ \ i,j\in\mathcal{I}^*.$$ Then it is not necessarily the case that there exist constants $\Lambda_{\ell m}$, $\Lambda_{\mu,m}$ with the property $\Lambda_{\ell m}+\Lambda_{\mu,m}=\tilde{\Lambda}_{\ell m}$ such that $\Lambda^T$ is a stability matrix of $(\mathcal{N}_{\mathcal{H}},X_{\mathcal{H}})$. Therefore, knowing that $\rho(\tilde{\Lambda}^T)<1$ does not give us any information on whether $(\mathcal{N}_{\mathcal{H}},X_{\mathcal{H}})$ has a globally attracting fixed point. This is essentially what is happening in example 3.

However, if the delayed dynamical network $(\mathcal{H},X^T)$ is non-distributed then it is then possible to relate the stability of $(\mathcal{N}_{\bar{\mathcal{H}}},X_{\bar{\mathcal{H}}})$ to the stability of $(\mathcal{N}_{\mathcal{H}},X_{\mathcal{H}})$. This is the main idea in the following proof of theorem \ref{theorem3.45}.

\begin{proof}
Let $(\mathcal{H},X^T)\in nd(X^T)$. Following the proof of theorem \ref{theorem3} we let $\tilde{\mathcal{I}}=\mathcal{I}\cup\mathcal{I}_{\mathcal{H}}$, $\mathcal{I}^*=\tilde{\mathcal{I}}-\{\ell m;11\}$, $\mu=\ell m;11$ and consider two cases.\\

\emph{Case 1:} For $\ell,m\in\mathcal{I}$ suppose $x_\ell^{k-1}$ is a variable of $\mathcal{H}_m$. Then the assumption that $(\mathcal{H},X^T)\in nd(X^T)$ implies $\mathcal{H}_m$ does not depend on $x_\ell^{k-0}$. Let $\bar{\mathcal{H}}_m$ be the function $\mathcal{H}_m$ in which the time-delayed variable $x_\ell^{k-1}$ is replaced by $x_\ell^{k-0}$ and define $\bar{\mathcal{H}}$ as in equation (\ref{eq8}).

Supposing $\tilde{\Lambda}^T$ is a stability matrix of $(\mathcal{N}_{\bar{\mathcal{H}}},X_{\bar{\mathcal{H}}})$ then
\begin{equation}\label{eq22.1}
d(\mathcal{N}_{\bar{\mathcal{H}}}(\textbf{x}|_{\mathcal{I}^*})_m,\mathcal{N}_{\bar{\mathcal{H}}} (\textbf{y}|_{\mathcal{I}^*})_m)\leq \sum_{i\in\mathcal{I}^*-\{\ell\}}\tilde{\Lambda}_{im}d(x_i,y_i)+\tilde{\Lambda}_{\ell m}d(x_\ell,x_m)
\end{equation}
for all $\textbf{x},\textbf{y}\in X_{\bar{\mathcal{H}}}$. Moreover, as $\mathcal{H}_m$ does not depend on $x_\ell^{k-0}$ then any stability matrix $\Lambda^T$ of $(\mathcal{N}_{\mathcal{H}},X_{\mathcal{H}})$ satisfies
\begin{equation}\label{eq22.2}
d(\mathcal{N}_{\mathcal{H}}(\textbf{x}|_{\tilde{\mathcal{I}}})_m,\mathcal{N}_{\mathcal{H}} (\textbf{y}|_{\tilde{\mathcal{I}}})_m)\leq \sum_{i\in\mathcal{I}^*-\{\ell\}}\Lambda_{im}d(x_i,y_i)+\Lambda_{\mu m}d(x_\mu,y_\mu)
\end{equation}
for all $\textbf{x},\textbf{y}\in X_{\mathcal{H}}$ since $\Lambda_{\ell m}$ can be assumed to be zero.

As $\big(\mathcal{N}_{\bar{\mathcal{H}}}\big)_m$ is the function $\big(\mathcal{N}_{\mathcal{H}}\big)_m$ in which the variable $x_{\ell m;11}$ is replaced by $x_\ell$ then equations (\ref{eq22.1}) and (\ref{eq22.2}) imply that the matrix
$$\Lambda_{ij}=\begin{cases}
\tilde{\Lambda}_{ij} &(i,j)\neq(\mu,m), \ (\ell,\mu), \ (\ell,m)\\
\tilde{\Lambda}_{\ell m} &(i,j)=(\mu,m)\\
1 &(i,j)=(\ell,\mu)\\
0 &(i,j)=(\ell,m)
\end{cases} \ \ \text{for} \ \ i,j\in\tilde{\mathcal{I}}$$
is a stability matrix of $(\mathcal{N}_{\mathcal{H}},X_{\mathcal{H}})$. By identifying the index $\mu$ with $0$ then the matrix $\tilde{\Lambda}_\theta(\alpha,\mathcal{L})=\Lambda$ for $\theta=1$, $\alpha=\Lambda_{\mu,m}$ and $\mathcal{L}=0$. As $\tilde{\Lambda}\geq 0$ then corollary \ref{corollary3} implies $\rho(\Lambda)<1$ if $\rho(\tilde{\Lambda})<1$. Hence, $\rho(\Lambda^T)<1$ if $\rho(\tilde{\Lambda}^T)<1$.

\emph{Case 2:} If $x_\ell^{k-\tau}$ is a variable of $\mathcal{H}_m$ for $\tau>1$ then again the assumption that $(\mathcal{H},X^T)$ is non-distributed implies that $\mathcal{H}_m$ does not depend on $x_\ell^{k-(\tau-1)}$. Let $\bar{\mathcal{H}}_m$ be the function $\mathcal{H}_m$ in which the time-delayed variable $x_\ell^{k-\tau}$ is replaced by $x_\ell^{k-(\tau-1)}$ and define $\bar{\mathcal{H}}$ as in equation (\ref{eq8}).

Using reasoning similar to that of the previous case where $\tau=1$ it can be shown that if $\tilde{\Lambda}^T$ is a stability matrix of $(\mathcal{N}_{\bar{\mathcal{H}}},X_{\bar{\mathcal{H}}})$ with $\rho(\tilde{\Lambda}^T)<1$ then there is a stability matrix $\Lambda^T$ of $(\mathcal{N}_{\mathcal{H}},X_{\mathcal{H}})$ such that $\rho(\Lambda^T)<1$.\\

Let $\tilde{A}$ be a stability matrix of $(\mathcal{U}_{\mathcal{H}},X)$. By proposition \ref{prop-2}, $\tilde{A}$ is then a stability matrix of $(\mathcal{U}_{\mathcal{H}},X)$ considered as a dynamical network without local dynamics. Proposition \ref{prop-1} moreover implies that $\tilde{A}=\tilde{\Lambda}^T$ where
$$d(\mathcal{U}_{\mathcal{H}}(\textbf{x}|_{\mathcal{I}})_j,\mathcal{U}_{\mathcal{H}} (\textbf{y}|_{\mathcal{I}})_j)\leq \sum_{i\in\mathcal{I}}\tilde{\Lambda}_{ij}d(x_i,y_i)$$
similar to equation (\ref{eq22.1}).

By sequentially replacing each variable $x_\ell^{k-\tau}$ of $(\mathcal{U}_{\mathcal{H}},X)$ by $x_\ell^{k-(\tau+1)}$ as in cases 1 and 2 it follows that if $\rho(\tilde{\Lambda}^T)<1$ then there exists a stability matrix $A$ of $(\mathcal{H},X^T)$ with the property $\rho(A)<1$. As the converse of this statement holds by theorem \ref{theorem3} this completes the proof.
\end{proof}

\section{Dynamical Network Expansions}

As mentioned in the introduction, a major obstacle in understanding the dynamics of a network (or high dimensional system) is that the information needed to do so is spread throughout the various system components. In the remainder of this paper we consider whether it is possible to transform a dynamical network while maintaining its global stability and consolidating this network information. Our main goal is to gain improve estimates of a network's global stability using our results on time-delayed dynamical networks.

\subsection{Expanding Dynamical Networks}
In this section we consider the method of \emph{dynamical network expansions} introduced in \cite{BW2012}. In section 6 this method is combined with the techniques developed in section 3 to greatly simplify this method and introduce the theory of \emph{dynamical network restrictions}. The results in this section are found in \cite{BW2012}.

Recall from section 2 that for a given dynamical network $(\mathcal{F},X)$ the component
$$\mathcal{F}_j:\bigoplus_{i\in\mathcal{I}_j}X_i\rightarrow X \ \ \text{for} \ \ 1\leq j\leq n.$$
Following \cite{BW2012} we can alternatively write
$$\mathcal{F}_j(\textbf{x}|_{\mathcal{I}_j})=\mathcal{F}_j(x_{j_1},\dots,x_{j_m}), \ \ \text{where} \ \ \mathcal{I}_j=\{j_1,\dots,j_m\}.$$
Note that if the variable $x_{j_i}$ of $\mathcal{F}_j$ is replaced by the function $\mathcal{G}(y_1,\dots,y_k)$ the result is the function
$$\mathcal{F}_j(x_{j_1},\dots,x_{j_{i-1}},\mathcal{G}(y_1,\dots,y_k),x_{j_{i+1}},\dots,x_{j_m})$$
having variables $\{x_{j_1},\dots,x_{j_{i-1}},y_1,\dots,y_k,x_{j_{i+1}},\dots,x_{j_m}\}$. Additionally, if the sequence $\underline{\gamma}=\ell_1,\dots,\ell_N$ let
$$\mathcal{F}_{j;\underline{\gamma}}(x_{j_1},\dots,x_{j_m})=\mathcal{F}_j(x_{j_1,\underline{\gamma}},\dots,x_{j_m,\underline{\gamma}}).$$
That is, the variables of the function $\mathcal{F}_{j;\underline{\gamma}}$ are indexed by the sequences $j_i;\ell_1,\dots,\ell_N$ for $1\leq i\leq m$.

\begin{definition}
For the graph $\Gamma_\mathcal{F}=(V,E)$, let $S\subseteq V$ such that each cycle of $\Gamma_\mathcal{F}$ contains a vertex in $S$.  If each vertex of $V$ belongs to a path or cycle that begins and ends with a vertex of $S$ then $S$ is called a \emph{complete structural set} of $\Gamma_\mathcal{F}$.
\end{definition}

For the graph $\Gamma_\mathcal{F}$ let $st_0(\Gamma_\mathcal{F})$ be the set of all complete structural sets of $\Gamma_\mathcal{F}$. For $S=\{v_1,\dots,v_m\}\in st_0(\Gamma_\mathcal{F})$ let $\mathcal{I}_S=\{1,\dots,m\}$ denote the \emph{index set} of $S$.

\begin{definition}
For $S\in st_0(\Gamma_\mathcal{F})$ and $i,j\in\mathcal{I}_S$ let $\mathcal{B}_{ij}(\Gamma_\mathcal{F};S)$ be the set of paths and cycles of $\Gamma_\mathcal{F}$ from $v_i$ to $v_j$ that have no interior vertices in $S$. The set $$\mathcal{B}_S(\Gamma_{\mathcal{F}})=\bigcup_{i,j\in\mathcal{I}_S}\mathcal{B}_{ij}(\Gamma_\mathcal{F};S)$$ is called the \emph{branch set} of $\Gamma_{\mathcal{F}}$ with respect to $S$.
\end{definition}

This notion of a branch set allows us to introduce the following.

\begin{definition}\label{admissable}
For $S\in st_0(\Gamma_{\mathcal{F}})$ let the set
\begin{equation}\label{eq.add}
\mathcal{A}_S(\mathcal{F})=\{\ell_1,\dots,\ell_N:v_{\ell_1},\dots,v_{\ell_N}\in\mathcal{B}_S(\Gamma_\mathcal{F}), \ N>2\}
\end{equation}
be the set of \textit{admissible sequences (paths)} of $\mathcal{F}$ with respect to $S$.
\end{definition}

Let $(\mathcal{F},X)$ be a dynamical network with graph of interactions $\Gamma_\mathcal{F}=(V,E)$ and suppose $S\in st_0(\Gamma_\mathcal{F})$. For $j\in\mathcal{I}_S$ let $\mathcal{F}_{\left<j,1\right>}$ be the function
$$\mathcal{F}_j=\mathcal{F}_j(x_{j_i},\dots,x_{j_m})$$
in which each variable $x_{j_\ell}$ is replaced by $x_{j_\ell,j}$ if $j_\ell\notin\mathcal{I}_S$.

For $i>1$ let $\mathcal{F}_{\left<j,i\right>}$ be the function
$$\mathcal{F}_{\left<j,i-1\right>}=\mathcal{F}_{\left<j,i-1\right>}(x_{\underline{\gamma}_1},\dots,x_{\underline{\gamma}_t})$$
in which each $x_{\underline{\gamma}_\ell}=x_{\ell_1,\dots,\ell_N}$ is replaced by the function $\mathcal{F}_{\ell_1;\underline{\gamma}_\ell}$ if $\ell_1\notin\mathcal{I}_S$. If $\ell_1\in\mathcal{I}_S$ for each $1\leq\ell\leq t$ then define $(\mathcal{X}_S\mathcal{F})_j=\mathcal{F}_{\left<j,{i-1}\right>}$.

Let $\underline{\gamma}=\ell_1,\dots,\ell_N\in\mathcal{A}_S(\mathcal{F})$. For $1<i<|\underline{\gamma}|=N$ define the $N-2$ spaces
$$X_{i;\underline{\gamma}}=X_{\ell_1}.$$
Additionally, define the functions
\begin{equation}\label{eq6}
\mathcal{X}_S\mathcal{F}_{i;\underline{\gamma}}(x_{i-1,\underline{\gamma}})=x_{i-1,\underline{\gamma}}.
\end{equation}
By way of notation we let
\begin{equation}\label{eq2.4}
X_{N-1,\underline{\gamma}}=X_{\underline{\gamma}}, \  \mathcal{X}_S\mathcal{F}_{N-1,\underline{\gamma}}=\mathcal{X}_S\mathcal{F}_{\underline{\gamma}}, \ \text{and} \ x_{1,\underline{\gamma}}=x_{\ell_1}.
\end{equation}

\begin{definition}\label{expansion}
Suppose $S\in st_0(\Gamma_\mathcal{F})$. Let
$$
\mathcal{X}_S\mathcal{F}=\Big(\bigoplus_{j\in\mathcal{I}_S}\mathcal{X}_S\mathcal{F}_j\Big) \oplus\Big(\bigoplus_{
\begin{smallmatrix}
\underline{\gamma}\in\mathcal{A}_S(\mathcal{F})\\
1<i<|\underline{\gamma}|
\end{smallmatrix}} \mathcal{X}_S\mathcal{F}_{i;\underline{\gamma}}\Big).$$
and
$$X_S=\Big(\bigoplus_{j\in\mathcal{I}_S}X_j\Big)\oplus\Big(\bigoplus_{
\begin{smallmatrix}
\underline{\gamma}\in\mathcal{A}_S(\mathcal{F})\\
1<i<|\underline{\gamma}|
\end{smallmatrix}}X_{i;\underline{\gamma}}\Big).$$
The dynamical network $(\mathcal{X}_S\mathcal{F},X_S)$ is called the \emph{expansion} of $(\mathcal{F},X)$ with respect to $S$.
\end{definition}

\begin{theorem}\label{gafp}
The dynamical network $(\mathcal{F},X)$ has a globally attracting fixed point if and only if the expansion also has $({\mathcal{X}_S\mathcal{F}},X_S)$ a globally attracting fixed point.
\end{theorem}

In \cite{BW2012} it is shown that if the expansion $({\mathcal{X}_S\mathcal{F}},X_S)$ has a globally attracting fixed point then the same holds for $(\mathcal{F},X)$. The converse holds by simply reversing the argument and so is omitted.

As there is no natural decomposition of the expansion $(\mathcal{X}_S\mathcal{F},X_S)$ into a set of local systems and an interaction we consider $(\mathcal{X}_S\mathcal{F},X_S)$ as a network with no local dynamics. By proposition \ref{prop-1} a stability matrix of $(\mathcal{X}_S\mathcal{F},X_S)$ therefore has the form $\Lambda^T$ where the constants $\Lambda_{ij}$ satisfy equation (\ref{eq2.3}) for $F=\mathcal{X}_S\mathcal{F}$. The following is an immediate corollary of theorem \ref{gafp} and theorem \ref{stability}.

\begin{corollary}\label{cor30}
Suppose $\Lambda^T$ is a stability matrix of the expansion $({\mathcal{X}_S\mathcal{F}},X_S)$. If $\rho(\Lambda^T)<1$ then $(\mathcal{F},X)$ has a globally attracting fixed point.
\end{corollary}

As has been shown, expansions always allow for better estimates (or at least no worse) of a dynamical network's global stability. This fact is a combination of theorem \ref{gafp} and the following.

\begin{theorem}\label{last}
Let $({\mathcal{X}_S\mathcal{F}},X_S)$ be an expansion of $(\mathcal{F},X)$ and suppose $\Lambda^T$ is a stability matrix of $(\mathcal{F},X)$. Then there is a stability matrix $\tilde{\Lambda}^T$ of $({\mathcal{X}_S\mathcal{F}},X_S)$ such that $\rho(\tilde{\Lambda}^T)\leq\rho(\Lambda^T)$.
\end{theorem}

\begin{figure}
  \begin{center}
    \begin{overpic}[scale=.5]{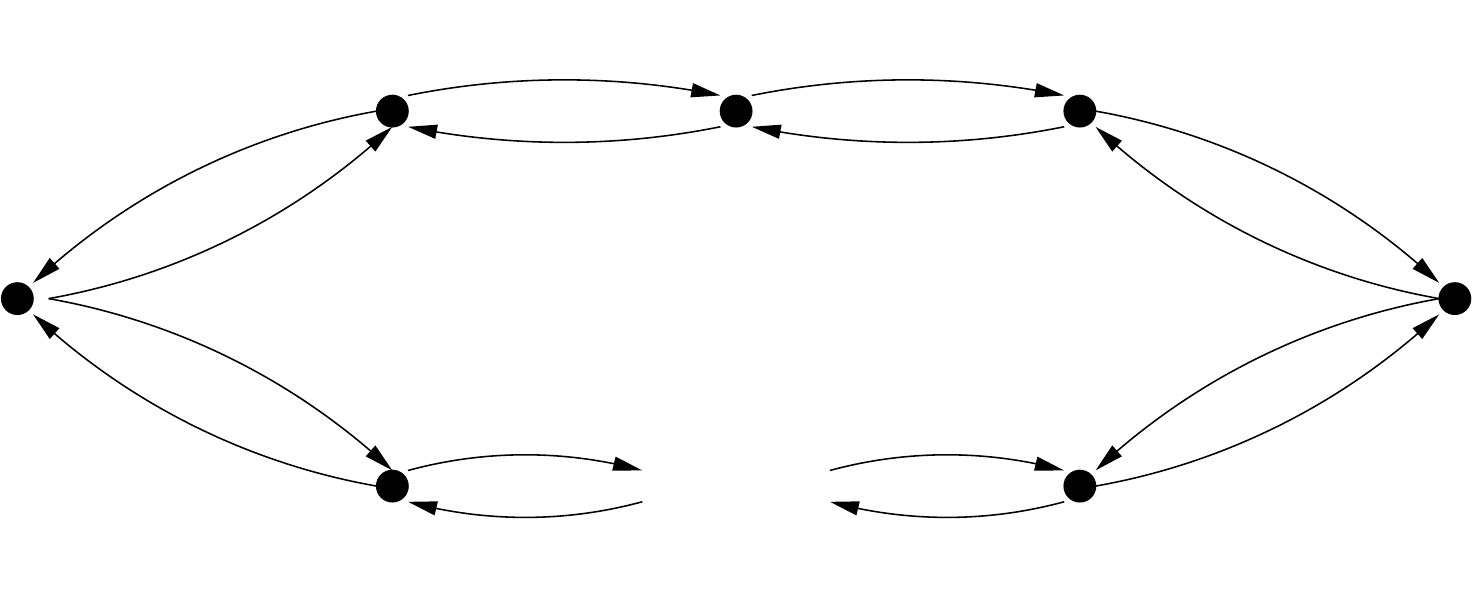}
    \put(49,-2){$\Gamma_F$}
    \put(48,37){$v_{2n}$}
    \put(25,37){$v_1$}
    \put(-7,19){$v_2$}
    \put(25,1){$v_3$}
    \put(70,1){$v_{2n-3}$}
    \put(102,19){$v_{2n-2}$}
    \put(70,37){$v_{2n-1}$}
    \put(47.5,7){$\dots$}
    \end{overpic}
  \end{center}
  \caption{The graph of interactions $\Gamma_\mathcal{F}$ for the dynamical network in example 5}\label{fig4}
\end{figure}
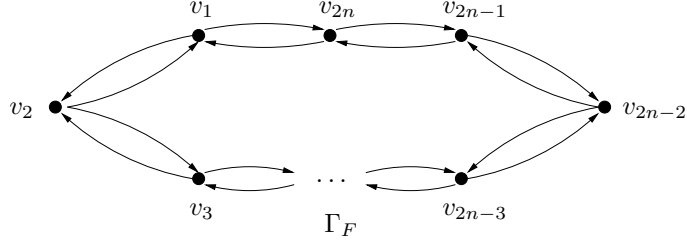

\begin{example}
Consider the dynamical network $(\mathcal{F},X)$ given by
$$\mathcal{F}_j(\mathbf{x})=\tanh(x_{j-1})+\tanh(x_{j+1})+c$$
with local systems $\varphi_i(x_i)=\tanh(x_i)$ and interaction
$$F_j(x_{j-1},x_{j+1})=x_{j-1}+x_{j+1}+c$$
for $1\leq i,j\leq 2n$ where the indices are taken mod $2n$, $X_i=\mathbb{R}$, and $c\geq0$. Note that this network is a variant of the Cohen-Grossberg model introduced in section 2.4 with $\mathcal{L}=1$, $c_j=c$,
$$W_{ij}=\begin{cases}
1 & j=i\pm 1\\
0 &\text{otherwise}
\end{cases}$$
and where we allow $\epsilon=1$. As $\mathcal{F}\in C_\infty^1(\mathbb{R}^{2n})$ then for $\Lambda_{ij}=\max_{\textbf{x}\in X}|(DF)_{ji}(\textbf{x})|$ and $L_i=\max_{\textbf{x}\in X}|\varphi_i^\prime(x_i)|$ the matrix $A=\Lambda^T\cdot diag[L_1,\dots,L_{2n}]$ given by
$$A=\left[
\begin{array}{ccccc}
0 & 1 & & & 1\\
1 & 0 & 1 & & \\
& \ddots & \ddots & \ddots &\\
& & 1 & 0 & 1\\
1 & & & 1 & 0\\
\end{array}
\right]$$
is a stability matrix of $(\mathcal{F},X)$. Since $A$ has the constant row sums $2$ then $\rho(\mathcal{F})=2$. Therefore, theorem \ref{stability} cannot directly be used to imply that $(\mathcal{F},X)$ has a globally attracting fixed point.

However, note that the vertex set $S=\{v_2,v_4,\dots,v_{2n}\}$ has the property that each cycle of $\Gamma_\mathcal{F}$ contains a vertex of $S$ (see figure \ref{fig4}). Moreover, as
\begin{equation}\label{branch}
\mathcal{B}_S(\Gamma_\mathcal{F})=\{v_i,v_j,v_k: i\in\mathcal{I}_S, \ j=i\pm 1, k=j\pm 1\}
\end{equation}
where each index is taken$\mod 2n$ then each vertex of $\Gamma_\mathcal{F}$ belongs to a path or cycle of $\mathcal{B}_S(\Gamma_\mathcal{F})$. Hence, $S$ is a complete structural set of $\Gamma_\mathcal{F}$.

Note that if $j \in\mathcal{I}_S$ then $j\pm 1\notin\mathcal{I}_S$. Replacing the variables $x_{j-1}$ and $x_{j+1}$ of $\mathcal{F}_j$ by $x_{j-1,j}$ and $x_{j+1,j}$ respectively results in the function
$$\mathcal{F}_{\left<j,1\right>}=\mathcal{F}_j\big(x_{j-1,j},x_{j+1,j}\big).$$

Given that the variables $x_{j-1,j}$ and $x_{j+1,j}$ of $\mathcal{F}_{\left<i,1\right>}$ are indexed by sequences beginning with elements not in $\mathcal{I}_S$ then they are replaced by $\mathcal{F}_{j-1;j-1,j}$ and $\mathcal{F}_{j+1;j+1,j}$ respectively to form $\mathcal{F}_{\left<j,2\right>}$. That is,
$$\mathcal{F}_{\left<j,2\right>}=\mathcal{F}_{j}\big(\mathcal{F}_{j-1}(x_{j-2,j-1,j},x_{j,j-1,j}), \mathcal{F}_{j+1}(x_{j,j+1,j},x_{j+2,j+1,j})\big).$$
As $j\in\mathcal{I}_S$ implies $j\pm2\in\mathcal{I}_S$ then each variable of $\mathcal{F}_{\left<j,2\right>}$ is indexed by a sequence beginning with an element of $\mathcal{I}_S$. Hence, $\mathcal{X}_S\mathcal{F}_j=\mathcal{F}_{\left<j,2\right>}$ for $j\in\mathcal{I}_S$. Moreover, it follows from (\ref{branch}) that
$$\mathcal{A}_S(\mathcal{F})=\{i,j,k:i\in\mathcal{I}_S,j=i\pm1,k=j\pm1\}$$
where each index is taken mod $2n$ and $n\geq 2$. For each $i,j,k=\underline{\gamma}\in\mathcal{A}_S(\mathcal{F})$ there is then a single function corresponding to $\underline{\gamma}$ where
$$\mathcal{X}_S\mathcal{F}_{2,\underline{\gamma}}:X_{1;\underline{\gamma}}\rightarrow X_{2;\underline{\gamma}} \ \ \text{given by} \ \ \mathcal{X}_S\mathcal{F}_{2,\underline{\gamma}}(x_{2;\underline{\gamma}})=x_{2;\underline{\gamma}}.$$ By use of (\ref{eq2.4}), this function can be written as $\mathcal{X}_S\mathcal{F}_{\underline{\gamma}}(x_i)=x_i$.

Following definition \ref{expansion} the expansion $\mathcal{X}_S\mathcal{F}:X_S\rightarrow X_S$ is given by

\begin{equation}\label{eq77}
\mathcal{X}_S\mathcal{F}(\mathbf{x})=
\left[
\begin{array}{c}
\mathcal{X}_{S}\mathcal{F}_{2}\\
\vdots\\
\mathcal{X}_{S}\mathcal{F}_{2n}
\end{array}
\right]
\bigoplus
\left[
\begin{array}{c}
\mathcal{X}_S\mathcal{F}_{2n,1,2}\\
\vdots\\
\mathcal{X}_S\mathcal{F}_{2,1,2n}
\end{array}
\right] \ \ \text{where}
\end{equation}

$$\mathcal{X}_S\mathcal{F}_{k,l,m}(x_k)=x_k \ \ \text{for} \ \ k,l,m\in\mathcal{A}_S(\mathcal{F}) \ \ \text{and}$$
\begin{align*}
\mathcal{X}_{S}\mathcal{F}_j=&\tanh\big[\tanh(x_{j-2,j-1,j})+\tanh(x_{j,j-1,j})+c\big]+\\ &\tanh\big[\tanh(x_{j,j+1,j})+\tanh(x_{j+2,j+1,j})+c\big]+c
\end{align*}
for $j\in\{2,4,\dots, 2n\}$. Moreover, the space
$$X_S=\Big(\bigoplus_{j=1}^n X_{2j}\Big)\oplus\Big(\bigoplus_{\underline{\gamma}\in\mathcal{A}_S(\mathcal{F})} X_{\underline{\gamma}}\Big)=\mathbb{R}^{5n}.$$ The graph of interactions $\Gamma_{\mathcal{X}_S\mathcal{F}}$ is shown in figure \ref{fig2001}.

\begin{figure}
  \begin{center}
    \begin{overpic}[height=1.5in,width=4in]{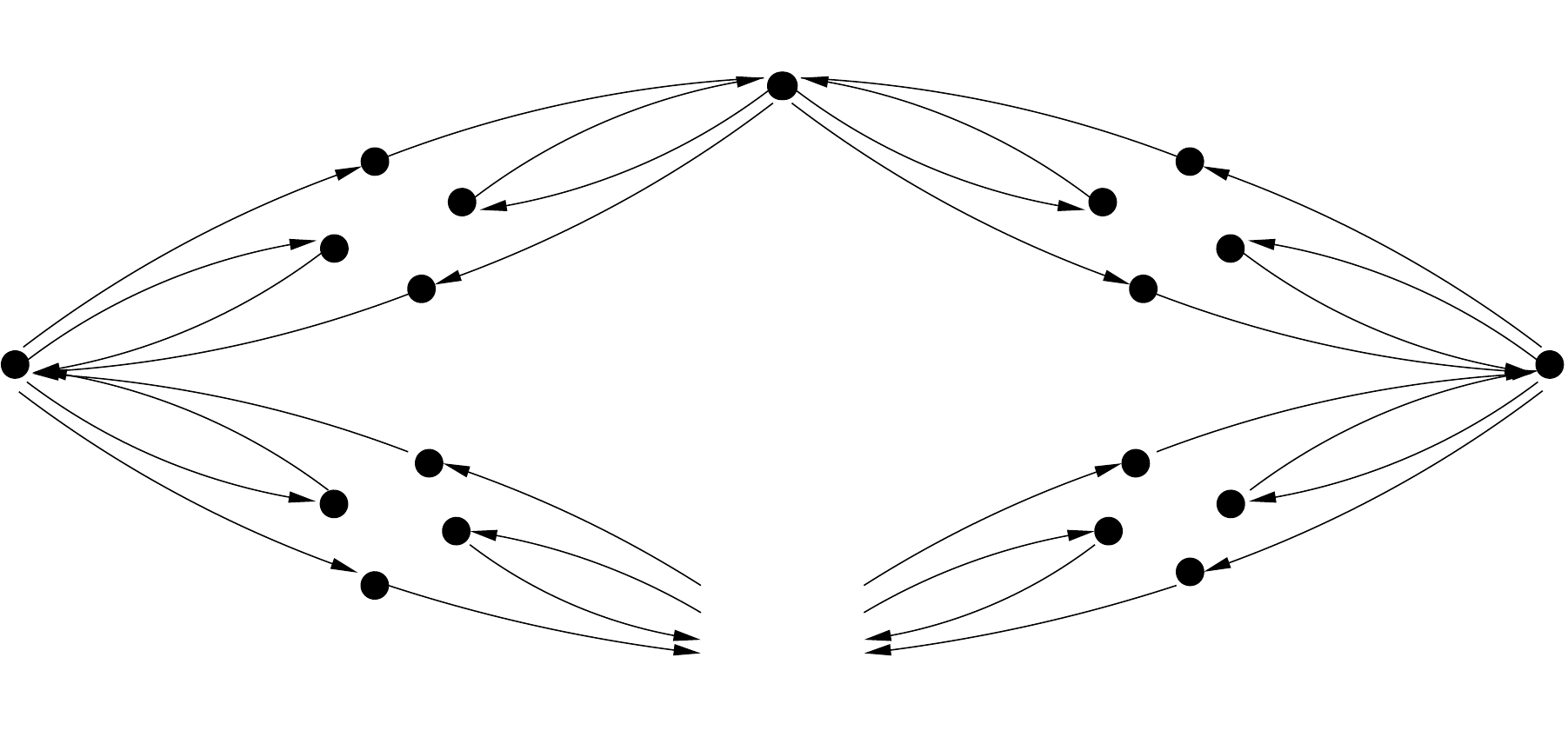}
    \put(46,-2){$\Gamma_{\mathcal{X}_SF}$}
    \put(48,35.5){$v_{m}$}

    \put(17,32){$v_{2;21m}$}
    \put(8,28){$v_{2;212}$}
    \put(23,19.5){$v_{2;m12}$}
    \put(34,23){$v_{2;m1m}$}

    \put(-5,18){$v_2$}

    \put(23,16.5){$v_{2;432}$}
    \put(8,8){$v_{2;232}$}
    \put(23,4){$v_{2;234}$}
    \put(34,13){$v_{2;434}$}

    \put(102,18){$v_{k}$}

    \put(67,16.5){$v_{2;ijk}$}
    \put(56,13){$v_{2;iji}$}
    \put(73,4){$v_{2;kji}$}
    \put(84,8){$v_{2;kjk}$}

    \put(73,32){$v_{2;klm}$}
    \put(84,28){$v_{2;klk}$}
    \put(67,19.5){$v_{2;mlk}$}
    \put(56,23){$v_{2;mlm}$}

    \put(48,6){$\dots$}
    \end{overpic}
  \end{center}
  \caption{The graph of the expansion $(\mathcal{X}_S\mathcal{F},X_S)$ in example 5 where $i=2n-4,j=2n-3,k=2n-2,l=2n-1,m=2n$}\label{fig2001}
\end{figure}
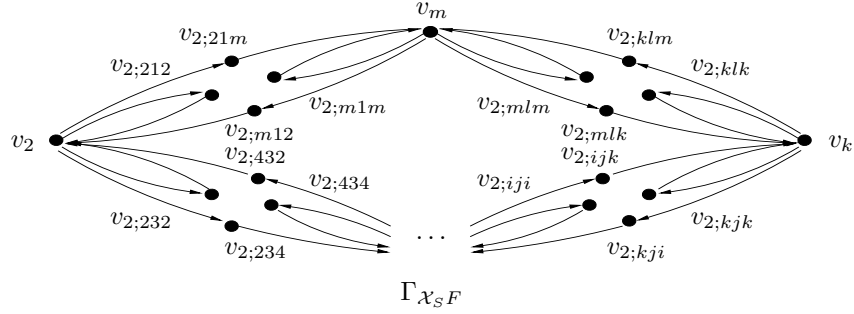

Let $\mathcal{A}_S\mathcal{F}_j=\{a,b,j\in\mathcal{A}_S\mathcal{F}\}$. For every $j\in\mathcal{I}_S$, $D\mathcal{X}_S\mathcal{F}$ has the form
$$(D\mathcal{X}_S\mathcal{F})_{ji}(\mathbf{x})=sech^2(x_i)\sech^2\big(c+\tanh(x_i)+\tanh(x_k)\big)$$
where $i,k\in\mathcal{A}_S\mathcal{F}_j$. Hence,
$\max_{\textbf{x}\in X_S}\left|(D\mathcal{X}_S\mathcal{F})_{ji}(\mathbf{x})\right|\leq\sech^2(c-2)$
for each $j\in\mathcal{I}_S$, $i\in\mathcal{A}_S\mathcal{F}_j$, and $c\geq 2$. One can compute that for $c\geq 2$,
$$\Lambda_{ij}=
\begin{cases}
\sech^2(c-2) \ \ \text{for} \ \ j\in\mathcal{I}_S,\  i\in\mathcal{A}_S(\mathcal{F})_j\\
\ \ 1 \ \ \ \ \ \ \ \ \ \ \ \ \ \text{for}  \ \ j=k,l,m; \ \ i=k\\
\ \ 0 \ \ \ \ \ \ \ \ \ \ \ \ \ \text{otherwise}
\end{cases}$$
gives a stability matrix $\Lambda^T$ of $(\mathcal{X}_S\mathcal{F},X_S)$.

To compute the spectral radius of $\Lambda^T$ let $\Lambda-\lambda I_{2n}$ be the block matrix
$$\Lambda-\lambda I_{2n}=\left[
\begin{array}{cc}
A & B\\
C & D
\end{array}
\right]$$
where $A\in\mathbb{R}^{n\times n}$ is the matrix with rows and columns indexed by $\mathcal{I}_S$. Hence, $A=\lambda I_n$ implying $A^{-1}=\lambda^{-1}I_n$ for $\lambda\neq0$.

Using the identity in equation (\ref{identity1}) it follows that $\det(\Lambda-\lambda I)=\det(\lambda^{-1}\tilde{\Lambda}-\lambda I)$ where the matrix $\tilde{\Lambda}\in\mathbb{R}^{n\times n}$ is given by
$$\tilde{\Lambda}=\left[
\begin{array}{ccccc}
2\sech^2(c-2) & \sech^2(c-2) & & &  \sech^2(c-2)\\
\sech^2(c-2) & \ddots & \ddots & &\\
& \ddots & & \ddots &\\
& & \ddots & \ddots &\sech^2(c-2)\\
\sech^2(c-2) & & & \sech^2(c-2) & 2\sech^2(c-2)\\
\end{array}
\right].$$
Since each row sum of $\tilde{\Lambda}$ is $4\sech^2(c-2)$ then $\rho(\tilde{\Lambda})=4\sech^2(c-2)$.

As $\det(\Lambda-\lambda I)=\det(\lambda^{-1}\tilde{\Lambda}-\lambda I)$ then $\det(\Lambda-\lambda I)=0$ only if $\det(\tilde{\Lambda}-\lambda^2 I)=0$ and so long as $\lambda\neq 0$. Hence, any nonzero $\lambda_0\in\mathbb{C}$ is an eigenvalue of $\tilde{\Lambda}$ only if $\pm\sqrt{|\lambda_0|}$ is an eigenvalue of $\Lambda$. Therefore,
$$\rho(\Lambda)=\sqrt{\rho(\tilde{\Lambda})}=2\sech(c-2).$$
As $2\sech(c-2)<1$ for $c>2+\sech^{-1}(1/2)\approx 3.31$ then $(\mathcal{F},X)$ does in fact have a globally attracting fixed point by theorem \ref{stability} for $c>2+\sech^{-1}(1/2)$.
\end{example}

In the next section we combine the results of section 4 on removing time delays with the theory of dynamical network expansions to greatly simplify the computational process in constructing and analyzing a dynamical network expansions.

\section{Restrictions of Dynamical Networks}
In the previous section we introduced the notion of a dynamical network expansion. In this section we show that any such expansion has the form of a delayed dynamical network. By removing these delays we will be able to develop a new procedure of dynamical network restrictions.

As with dynamical network expansions such restrictions allow for improved stability estimates of a dynamical network's global stability. However, the computational effort needed to carry out this restrictions procedure is far less when compared with the procedure of constructing and analyzing expansions in section 5.

We note that restrictions are somewhat similar to the dynamical network reductions introduced in \cite{BW2012}. The fundamental difference though is that dynamical network reductions preserve the spectrum of a dynamical network whereas the restrictions we introduce involve the removal of time-delays and so do not have this property.

\subsection{Restricting Dynamical Networks}
Let $(\mathcal{F},X)$ be a dynamical network and suppose $S\in st_0(\Gamma_\mathcal{F})$. For $j\in\mathcal{I}_S$ let $\mathcal{F}_{(j,0)}=\mathcal{F}$. For $i\geq1$ define $\mathcal{F}_{(j,i)}$ to be the function
$$\mathcal{F}_{(j,i-1)}=\mathcal{F}_{(j,i-1)}(x_{j_1},\dots,x_{j_m})$$
in which each variable $x_{j_\ell}$ is replaced by the function $\mathcal{F}_{j_\ell}$ if $j_\ell\notin\mathcal{I}_S$. If $j_\ell\in\mathcal{I}_S$ for each $1\leq\ell\leq m$ then let $\mathcal{R}_S\mathcal{F}_j=\mathcal{F}_{(j,i-1)}$. Note that if $\mathcal{R}_S\mathcal{F}_j=\mathcal{F}_{(j,i-1)}(x_1,\dots,x_m)$ then $\{1,\dots,m\}\subseteq\mathcal{I}_S$.

\begin{definition}\label{reduction}
Suppose $S\in st_0(\Gamma_\mathcal{F})$. Let $X|_S=\bigoplus_{i\in\mathcal{I}_S}X_i$ and define
$$
\mathcal{R}_S\mathcal{F}=\bigoplus_{j\in\mathcal{I}_S}\mathcal{R}_S\mathcal{F}_j.$$
The dynamical network $(\mathcal{R}_S\mathcal{F},X|_S)$ is called the \emph{restriction} of the dynamical network $(\mathcal{F},X)$ to $S$.
\end{definition}

As is the case with dynamical network expansions we consider the restriction $(\mathcal{R}_S\mathcal{F},X|_S)$ to be a dynamical network without local dynamics. We also note that the restriction $(\mathcal{R}_S\mathcal{F},X|_S)$ is simpler than the original dynamical network in at least two ways. The first is that $(\mathcal{R}_S\mathcal{F},X|_S)$ is a lower dimensional system than $(\mathcal{F},X)$. The second is that the graph $\Gamma_{\mathcal{R}_S\mathcal{F}}$ structurally simpler having less vertices and edges than $\Gamma_{\mathcal{F}}$ and potentially less paths and cycles, etc.

\begin{example}
Consider the dynamical network $(\mathcal{F},X)$ of the form
$$\mathcal{F}(\textbf{x})=\left[
\begin{array}{l}
\mathcal{F}_1(x_6)\\
\mathcal{F}_2(x_1)\\
\mathcal{F}_3(x_2,x_5,x_6)\\
\mathcal{F}_4(x_3)\\
\mathcal{F}_5(x_2,x_3,x_4)\\
\mathcal{F}_6(x_5)
\end{array}\right].$$
Note that $S=\{v_1,v_3,v_5\}$ is a complete structural set of the graph of interactions $\Gamma_{\mathcal{F}}$ shown in figure \ref{fig6.1} (left). Replacing each of the variables $x_i$ by $\mathcal{F}_{i}$ in each $\mathcal{F}_j$ for $i\notin \mathcal{I}_S$ and $j\in \mathcal{I}_S$ yields the functions
\begin{align*}
\mathcal{F}_{(1,1)}&=\mathcal{F}_1\big(\mathcal{F}_6(x_5)\big),\\ \mathcal{F}_{(2,1)}&=\mathcal{F}_3\big(\mathcal{F}_2(x_1),x_5,\mathcal{F}_6(x_5)\big),\\ \mathcal{F}_{(3,1)}&=\mathcal{F}_5\big(\mathcal{F}_2(x_1),x_3,\mathcal{F}_4(x_3)\big).
\end{align*}
\begin{figure}
  \begin{center}
    \begin{overpic}[scale=.4]{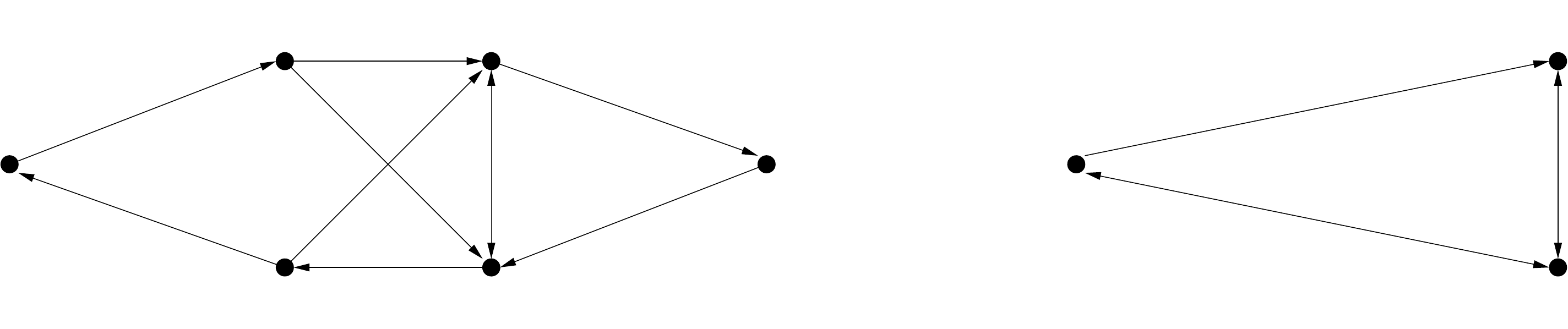}
    \put(24,-2){$\Gamma_{\mathcal{F}}$}
    \put(-3,10){$v_1$}
    \put(17,19){$v_2$}
    \put(30,19){$v_3$}
    \put(50,10){$v_4$}
    \put(30,1){$v_5$}
    \put(17,1){$v_6$}
    \put(81,-2){$\Gamma_{\mathcal{R}_S\mathcal{F}}$}
    \put(65,10){$v_1$}
    \put(101,17){$v_3$}
    \put(101,3){$v_5$}
    \end{overpic}
  \end{center}
  \caption{The graph of interactions of $(\mathcal{F},X)$ and its restriction $(\mathcal{R}_S\mathcal{F},X|_S)$ in example 6}\label{fig6.1}
\end{figure}
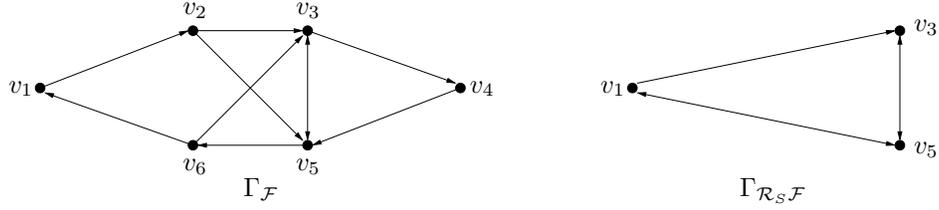
Since each variable of $\mathcal{F}_{(1,1)}$, $\mathcal{F}_{(2,1)}$, and $\mathcal{F}_{(3,1)}$ is indexed by an element of $\mathcal{I}_S$ then $\mathcal{R}_S\mathcal{F}:X|_S\rightarrow X|_S$ is given by
$$\mathcal{R}_S\mathcal{F}(\textbf{x})=\left[
\begin{array}{l}
\mathcal{F}_1\big(\mathcal{F}_6(x_5)\big)\\
\mathcal{F}_3\big(\mathcal{F}_2(x_1),x_5,\mathcal{F}_6(x_5)\big)\\
\mathcal{F}_5\big(\mathcal{F}_2(x_1),x_3,\mathcal{F}_4(x_3)\big)
\end{array}\right].$$
The graph of interactions $\Gamma_{\mathcal{R}_S\mathcal{F}}$ is shown in figure \ref{fig6.1} (right).
\end{example}

\begin{theorem}\label{theorem5}
Suppose $A$ is a stability matrix of the dynamical network $(\mathcal{F},X)$. If $\rho(A)<1$
then any restriction $(\mathcal{R}_S\mathcal{F},X|_S)$ has a globally attracting fixed point.
\end{theorem}

That is, if a dynamical network is known to be stable by theorem \ref{stability} the same then is true of any of its restrictions. However, if the restriction $(\mathcal{R}_S\mathcal{F},X|_S)$ is known to be stable it is not always the case that the unrestricted system $(\mathcal{F},X)$ has a globally attracting fixed point. To determine when this is the case we introduce the following concept.

\begin{definition}
Let $S$ be a complete structural set of $\Gamma_{\mathcal{F}}$. Then $S$ is called a \emph{basic structural set} of $\Gamma_{\mathcal{F}}$ if
$$|\mathcal{B}_{ij}(\Gamma_{\mathcal{F}};S)|\leq 1 \ \ \text{for all} \ \ i,j\in\mathcal{I}_S.$$
\end{definition}
If $S$ is a basic structural set of $\Gamma_{\mathcal{F}}$ we write $S\in bas(\Gamma_\mathcal{F})$. The notions of a basic structural set and a non-distributed interaction are related and allow us to formulate and prove the following theorem.

\begin{theorem}\label{theorem6}
Suppose $S\in bas(\Gamma_{\mathcal{F}})$. If $A$ is a stability matrix of $(\mathcal{F},X)$ with the property $\rho(A)<1$ then $(\mathcal{R}_S\mathcal{F},X|_S)$ has a globally attracting fixed point. Conversely, if $\tilde{A}$ is a stability matrix of $(\mathcal{R}_S\mathcal{F},X|_S)$ with $\rho(\tilde{A})<1$ then $(\mathcal{F},X)$ has a globally attracting fixed point.
\end{theorem}

Before giving the proof of theorem \ref{theorem5} and \ref{theorem6} we consider the following example.

\begin{example}
As in example 5 let $(\mathcal{F},X)$ be the dynamical network given by
$$\mathcal{F}_j(\mathbf{x})=\tanh(x_{j-1})+\tanh(x_{j+1})+c$$
for $1\leq j\leq 2n$ where the indices are taken mod $2n$, $X_i=\mathbb{R}$, and $c\geq0$. As $S=\{v_2,v_4,\dots,v_{2n}\}$ is a complete structural set of $\Gamma_\mathcal{F}$ then the dynamical network restriction $(\mathcal{R}_S\mathcal{F},X|_S)$ is defined and is constructed as follows.

For $j\in\mathcal{I}_S$ the component function $\mathcal{F}_j=\mathcal{F}_{(j,0)}$ has variables $x_{j-1}$ and $x_{j+1}$ where both $j\pm 1\notin\mathcal{I}_S$. The function $\mathcal{F}_{(j,1)}$ is then
\begin{align*}
\mathcal{F}_{(j,1)}(x_{j-2},x_j,x_{j+2})=&\tanh\big[\tanh(x_{j-2})+\tanh(x_{j})+c\big]+\\ &\tanh\big[\tanh(x_{j})+\tanh(x_{j+2})+c\big]+c
\end{align*}
Since each variable of $\mathcal{F}_{(j,1)}$ is indexed by an element of $\mathcal{I}_S$ then $\mathcal{F}_{(j,1)}=\mathcal{R}_S\mathcal{F}_j$. Therefore, the dynamical network restriction  $(\mathcal{R}_S\mathcal{F},X|_S)$ is given by
$$\mathcal{R}_S\mathcal{F}=\bigoplus_{j=1}^n\mathcal{R}_S\mathcal{F}_{2j} \ \ \text{and} \ \ X|_S=\mathbb{R}^{n}.$$
The graph of interactions $\Gamma_{\mathcal{R}_S\mathcal{F}}$ is shown in figure \ref{fig6}. Moreover, as
$$|\mathcal{B}_{ij}(\Gamma_{\mathcal{F}})|=\begin{cases}
1 & j=i\pm2\\
1 & j=i\\
0 & otherwise
\end{cases} \ \ \text{for} \ \ i,j\in\mathcal{I}_S$$
then $S\in bas(\Gamma_{\mathcal{F}})$. As in example 5 one can compute that
$$|D\mathcal{R}_S\mathcal{F}_{ji}(\textbf{x})|\leq
\begin{cases}
2\sech^2(c-2) \ &\text{for} \ \ i=j\\
\sech^2(c-2) \ &\text{for} \ \ i=j\pm 2\\
0 &\text{otherwise}
\end{cases}
$$
for $j\in\mathcal{I}_S$ and $c\geq 2$. Hence, the matrix $\tilde{\Lambda}$ given in example 5 is a stability matrix of $(\mathcal{R}_S\mathcal{F},X|_S)$. Since $\rho(\tilde{\Lambda})=4\sech^2(c-2)$ then as in example 5 theorem \ref{theorem6} implies that the original unrestricted network $(\mathcal{F},X)$ has a globally attracting fixed point if $c>2+\sech^{-1}(1/2)$.

Importantly, if we were to analyze $(\mathcal{F},X)$ as a Cohen-Grossberg network using theorem \ref{CGtheorem} (allowing $\epsilon=1$) we could not deduce the stability of the system for any value of $c$. In fact, the stability criteria in theorem \ref{CGtheorem} does not depend on the constants $c_j$. However, by use of a dynamical network restriction we gain more information about the stability of such networks. Specifically, it follows that the stability of the Cohen-Grossberg networks considered in section 2.2 depend on $\mathcal{L}$, $\rho(|W|)$, and the values of $c_j$ for $\epsilon\approx 1$.

It is also worth mentioning again that both the process of constructing $(\mathcal{R}_S\mathcal{F},X|_S)$ and analyzing its stability are significantly simpler than finding and analyzing the expansion $(\mathcal{X}_S\mathcal{F},X_S)$ in example 5. (As evidence, one can compare the graphs $\Gamma_{\mathcal{X}_S\mathcal{F}}$ and $\Gamma_{\mathcal{R}_S\mathcal{F}}$ in figures 7 and 9.) This should not be surprising based on how expansions and restrictions are defined. In fact, for any dynamical network $(\mathcal{F},X)$ it is always computationally easier to analyze the restriction $(\mathcal{R}_S\mathcal{F},X|_S)$ compared with the expansion $(\mathcal{X}_S\mathcal{F},X_S)$.

\begin{figure}
  \begin{center}
    \begin{overpic}[scale=.5]{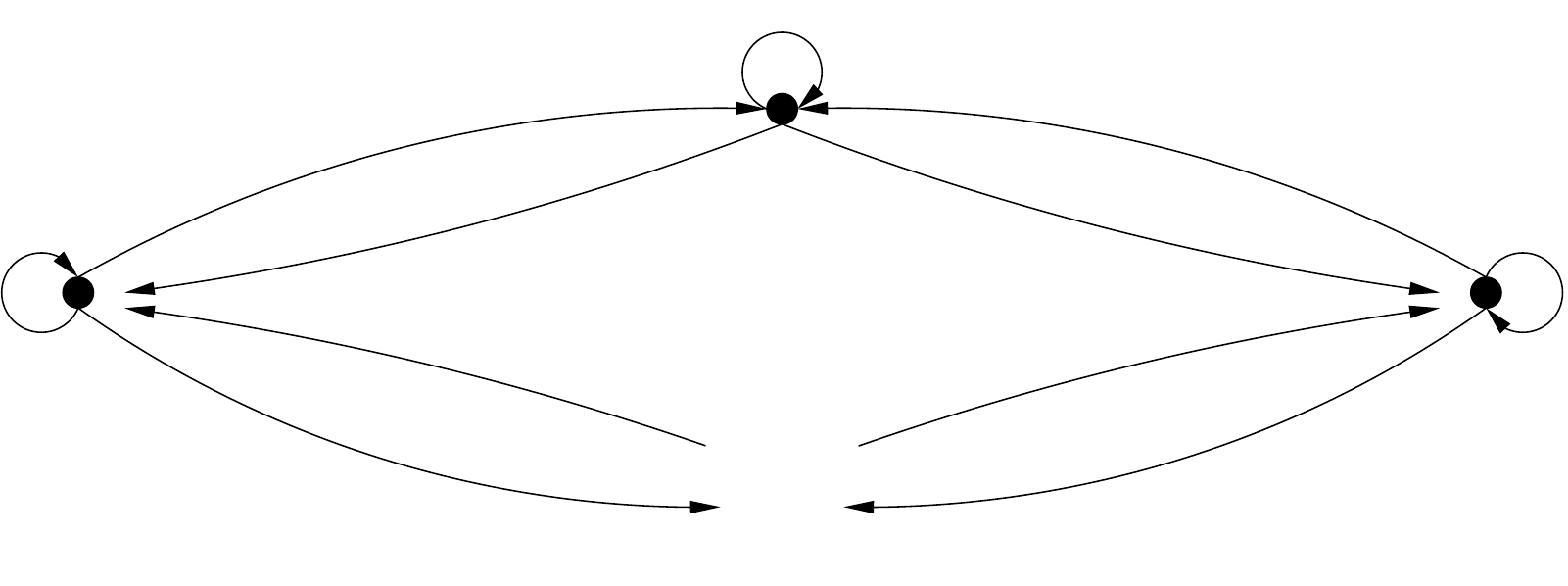}
    \put(44,-3){$\Gamma_{\mathcal{R}_S\mathcal{F}}$}
    \put(47,37){$v_{2n}$}
    \put(-6,18){$v_2$}
    \put(101,18){$v_{2n-2}$}
    \put(47.5,7){$\dots$}
    \end{overpic}
  \end{center}
  \caption{}\label{fig6}
\end{figure}
\end{example}

In examples 5 and 7 both expansions and restrictions were used to gain improved stability estimates of the untransformed dynamical network $(\mathcal{F},X)$. Although using the restrictions $(\mathcal{R}_S\mathcal{F},X|_S)$ was computationally and procedurally simpler than using $(\mathcal{X}_S\mathcal{F},X_S)$ both gave the same improved stability estimate. This outcome is not a coincidence but holds in general.

\begin{theorem}\label{theorem4}
Suppose $S\in bas(\Gamma_{\mathcal{F}})$. There is a stability matrix $\Lambda^T$ of  $(\mathcal{X}_S\mathcal{F},X_S)$ with property $\rho(\Lambda^T)<1$ if and only if there is a stability matrix $\tilde{\Lambda}^T$ of $(\mathcal{R}_S\mathcal{F},X|_S)$ such that $\rho(\tilde{\Lambda}^T)<1$.
\end{theorem}

As a final observation, in this section, we note that it is often possible to sequentially restrict a dynamical network and thereby sequentially improve ones estimate of whether the original (untransformed) network has a globally attracting fixed point. Moreover, it is possible to use restrictions to get improved stability estimates of the time-delayed network $(\mathcal{H},X^T)$ by restricting the undelayed system $(\mathcal{U}_{\mathcal{H}},X)$.

\subsection{Proof of Theorems \ref{theorem5}, \ref{theorem6}, and \ref{theorem4}}
The major idea needed to prove the theorems of the previous section is that any dynamical network expansion is related to a time-delayed dynamical network. More specifically, for the expansion $(\mathcal{X}_S\mathcal{F},X_S)$ any admissible sequence $\underline{\gamma}=0,\dots,\tau-1\in\mathcal{A}_S(\mathcal{F})$ corresponds to the functions
\begin{equation}\label{eq4}
\mathcal{X}_S\mathcal{F}_{i;\underline{\gamma}}(x_{i-1,\underline{\gamma}})=x_{i-1,\underline{\gamma}}
\end{equation}
for $1<i<|\underline{\gamma}|$. That is, after $\tau-1$ iterates the function $\mathcal{X}_S\mathcal{F}_{\tau-1}$ depends on the value of $x_0$. Therefore, every admissible sequence $\underline{\gamma}\in\mathcal{A}_S(\mathcal{F})$ corresponds to a time delay. By formally introducing delays into the expansion we can relate the stability of $(\mathcal{X}_S\mathcal{F},X_S)$ to that of $(\mathcal{R}_S\mathcal{F},X|_S)$ using techniques from sections 3 and 4. This can be done as follows.

For $\underline{\nu}=0,\dots,\tau-1\in\mathcal{A}_S(\mathcal{F})$ let
$$\mathcal{A}_{\underline{\nu}}(\mathcal{F})=\mathcal{A}_S(\mathcal{F})-\underline{\nu} \ \ \text{and} \ \ \mathcal{I}_{\underline{\nu}}=(\mathcal{I}_S-\{\tau-1\})\cup\underline{\nu}.$$
Let $\mathcal{X}_S\mathcal{F}_{\underline{\nu}}$ be the function $\mathcal{X}_S\mathcal{F}_{\tau-1}$ in which the variable $x_{\underline{\nu}}$ is replaced by the time delayed variable $x_0^{k-(\tau-1)}$.

Observe that the map $\mathcal{D}_{\underline{\nu}}\big(\mathcal{X}_S\mathcal{F}\big):X^\tau_{\underline{\nu}}\rightarrow X_{\underline{\nu}}$ given by
$$\mathcal{D}_{\underline{\nu}}\big(\mathcal{X}_S\mathcal{F}\big)= \Big(\bigoplus_{j\in\mathcal{I}_{\underline{\nu}}}\mathcal{X}_S\mathcal{F}_j\Big) \oplus\Big(\bigoplus_{
\begin{smallmatrix}
\underline{\gamma}\in\mathcal{A}_{\underline{\nu}}(\mathcal{F})\\
1<i<|\underline{\gamma}|
\end{smallmatrix}} \mathcal{X}_S\mathcal{F}_{i;\underline{\gamma}}\Big) \ \ \text{with}$$
$$X_{\underline{\nu}}=\Big(\bigoplus_{j\in\mathcal{I}_{\underline{\nu}}}X_j\Big)\oplus\Big(\bigoplus_{
\begin{smallmatrix}
\underline{\gamma}\in\mathcal{A}_{\underline{\nu}}(\mathcal{F})\\
1<i<|\underline{\gamma}|
\end{smallmatrix}}X_{i;\underline{\gamma}}\Big)$$
defines the time-delayed dynamical network $(\mathcal{D}_{\underline{\nu}}\big(\mathcal{X}_S\mathcal{F}\big),X_{\underline{\nu}}^\tau)$. By identifying the index $0,\tau-1;i-1,\tau-2\in\mathcal{I}_{\mathcal{D}_{\underline{\nu}}(\mathcal{X}_S\mathcal{F})}$ with the index $i;\underline{\nu}$ for $1<i\leq\tau-1$ then $(\mathcal{N}_{\mathcal{D}_{\underline{\nu}}(\mathcal{X}_S\mathcal{F})},X_S)=(\mathcal{X}_S\mathcal{F},X_S)$. This follows from the fact that the functions
\begin{align*}
(\mathcal{N}_{\mathcal{D}_{\underline{\nu}}(\mathcal{X}_S\mathcal{F})})_{0,\tau-1;i-1,\tau-2}(x_{0,\tau-1;i-2,\tau-2})&= x_{0,\tau-1;i-2,\tau-2} \ \ \text{and}\\
\mathcal{X}_S\mathcal{F}_{i;\underline{\nu}}(x_{i-1;\underline{\nu}})&=x_{i-1;\underline{\nu}}
\end{align*}
have the same form (see equations (\ref{eq5}) and (\ref{eq6})). By sequentially modifying the expansion $(\mathcal{X}_S\mathcal{F},X_S)$ in this manner over all admissible sequences in $\mathcal{A}_S(\mathcal{F})$ the result is the time-delayed dynamical network
$$\big(\mathcal{D}_{\mathcal{A}_S(\mathcal{F})}(\mathcal{X}_S\mathcal{F}),X_{\mathcal{A}_S(\mathcal{F})}^T\big) \ \ \text{where} \ \ T=\max_{\underline{\gamma}\in\mathcal{A}_S(\mathcal{F})}|\underline{\gamma}|-2.$$
For simplicity we let $\mathcal{D}_{\mathcal{A}_S(\mathcal{F})}(\mathcal{X}_S\mathcal{F})=\mathcal{D}_S\mathcal{F}$ and note that the product space $X_{\mathcal{A}_S(\mathcal{F})}^T=X|_S^T$. Hence, the dynamical network $$(\mathcal{N}_{\mathcal{D}_S\mathcal{F}},X_S)=(\mathcal{X}_S\mathcal{F},X_S)$$
by identifying the index $\ell_1,\ell_N;i-1,\ell_N-1\in\mathcal{I}_{\mathcal{D}_S\mathcal{F}}$ with the index $i;\underline{\gamma}$ for all $\underline{\gamma}=\ell_1,\dots,\ell_N\in\mathcal{A}_S(\mathcal{F})$. The following is then a result of theorem \ref{next}.

\begin{lemma}\label{lemma6}
Let $S\in st_0(\Gamma_{\mathcal{F}})$. Then $(\mathcal{D}_{S}\mathcal{F},X|_S^T)$ has a globally attracting fixed point if and only if the same is true of the expansion $(\mathcal{X}_S\mathcal{F},X_S)$.
\end{lemma}

Moreover, note that for $j\in\mathcal{I}$ by replacing each time delayed variable $x^{k-\ell}_i$ of $\mathcal{D}_S\mathcal{F}_j$ by $x_i$ the result is the function $\mathcal{R}_S\mathcal{F}_j$. Therefore,
\begin{equation}\label{eq7}
(\mathcal{U}_{\mathcal{D}_S\mathcal{F}},X|_S)=(\mathcal{R}_S\mathcal{F},X|_S).
\end{equation}
We now give a proof of theorem \ref{theorem5}.

\begin{proof}
Suppose $A$ is a stability matrix of $(\mathcal{F},X)$ with $\rho(A)<1$. Hence, $(\mathcal{F},X)$ has a globally attracting fixed point. Theorem \ref{gafp} together with lemma \ref{lemma6} then imply that $(\mathcal{D}_{S}\mathcal{F},X|_S^T)$ also has a globally attracting fixed point. The undelayed dynamical network $(\mathcal{U}_{\mathcal{D}_{S}\mathcal{F}},X|_S)$ is then globally stable by theorem \ref{theorem3} and the result follows from equation (\ref{eq7}) as $(\mathcal{U}_{\mathcal{D}_S\mathcal{F}},X|_S)=(\mathcal{R}_S\mathcal{F},X|_S)$.
\end{proof}

A proof of theorem \ref{theorem6} is the following.

\begin{proof}
Suppose $S\in bas(\Gamma_{\mathcal{F}})$. If $A$ is a stability matrix of $(\mathcal{F},X)$ with $\rho(A)<1$ then theorem \ref{theorem5} implies $(\mathcal{R}_S\mathcal{F},X|_S)$ has a globally attracting fixed point. Therefore, suppose $\tilde{A}$ is a stability matrix of $(\mathcal{R}_S\mathcal{F},X|_S)$ with the property $\rho(\tilde{A})<1$.

Since $S\in bas(\Gamma_{\mathcal{F}})$ then it follows that $(\mathcal{D}_{S}\mathcal{F},X|_S^T)\in nd(X|_S^T)$. As the undelayed network $(\mathcal{U}_{\mathcal{D}_S\mathcal{F}},X|_S)=(\mathcal{R}_S\mathcal{F},X|_S)$ then theorem \ref{theorem3.45} implies that there is a stability matrix $A$ of $(\mathcal{U}_{\mathcal{D}_S\mathcal{F}},X|_S)$ such that $\rho(A)<1$. Lemma \ref{lemma6} combined with theorem \ref{gafp} then imply that the dynamical network $(\mathcal{F},X)$ has a globally attracting fixed point.
\end{proof}

Note that in the proof of theorem \ref{theorem4} the assumption that $S\in bas(\Gamma_{\mathcal{F}})$ implies that the network $(\mathcal{D}_{S}\mathcal{F},X|_S^T)\in nd(X|_S^T)$. Since $(\mathcal{U}_{\mathcal{D}_S\mathcal{F}},X|_S)=(\mathcal{R}_S\mathcal{F},X|_S)$ then theorem \ref{theorem3.45} directly implies the result of theorem \ref{theorem4}.

\section{Concluding Remarks}
This paper continues the analysis of useful transformations of dynamical networks initiated in \cite{BW2012}. Here, we analyze time-delayed dynamical networks and their global stability. Because analyzing the global stability of a network requires knowledge of its spectral radius, rather than the knowledge of its entire spectrum, we are able to introduce a new much simpler class of dynamical network transformations, which we call network restrictions. Such restrictions allow us to study the global stability of dynamical networks in a simpler and computationally more efficient way compared with the isospectral networks expansions found in \cite{BW2012}. However, we note that as such expansions preserve the entire spectrum of the network they provide more information about the dynamical network.

A key ingredient in our procedure of dynamical network restrictions is the notion of a basic structural set of the network's graph of interactions. These sets form a subclass of structural sets introduced in \cite{BW2012}. Our approach here allows us to prove that dynamical networks and their time-delayed versions with non-distributed delays are globally stable or unstable simultaneously. We observe that this is not true in general for dynamical networks with distributed delays.

The theory developed in this paper is illustrated by various examples of Cohen-Grossberg neural networks. Importantly, this approach to analyzing the global stability of networks is not limited in any way to this class of networks and could be applied to any (time-delayed or undelayed) dynamical network. We fully expect that this theory will prove to be useful for the analysis of other features, structures, and dynamics of networks.


\begin{thebibliography}{9}
\bibitem{Afriamovich07} V. Afraimovich and L. Bunimovich. Dynamical networks: interplay of topology, interactions, and local dynamics, \textit{Nonlinearity} \textbf{20} (2007) 1761-1771.

\bibitem{Albert02} R. Albert and A.-L. Barab\'{a}si, Statistical mechanics of complex networks \emph{Rev. Mod. Phys.} \textbf{74} (2002) 47-97.

\bibitem{BW2012} L. Bunimovich and B. Webb, Isospectral graph transformations, spectral equivalence, and global stability of dynamical networks. \textit{Nonlinearity} \textbf{25} (2012) 211-254.

\bibitem{MCohen1983} M. Cohen and S. Grossberg, Absolute stability and global pattern formation and
parallel memory storage by competitive neural networks, IEEE \emph{Transactions on Systems, Man, and Cybernetics} SMC-13 (1983) 815–821.

\bibitem{JCao2005} J. Cao and W. Xiong, Global exponential stability of discrete-time Cohen-Grossberg neural networks. \emph{Neurocomputing} \textbf{64} (2005) 433–446.

\bibitem{Dorogovtsev03} A. Dorogovtsev and J. Mendes, Evolution of Networks: From Biological Networks to the Internet and WWW Oxford: Oxford Univ. Press (2003).

\bibitem{Faloutsos99} M. Faloutsos, P. Faloutsos and C. Faloutsos, On power-law relationship of the internet topology \textit{ACMSIGCOMM,`99, Comput. Commun. Rev.} \textbf{29} (1999) 251-263.

\bibitem{MGupta1994} M. Gupta, L. Jin, and P. Nikiforuk, Absolute stability conditions for discrete-time recurrent neural networks. IEEE \emph{Transactions on Neural Networks}, Vol. 5, No. 6, (1994) 954-964.

\bibitem{Horn85} R. Horn and C. Johnson, Matrix Analysis, Cambridge University Press, Cambridge, (1990).

\bibitem{Newman06} M. Newman, A.-L. Barab\'{a}si, and D. Watts (ed), The Structure of Dynamic Networks, Princeton: Princeton Univ. Press (2006).

\bibitem{CLi2010} C. Li, X. Liao, and S. Zhong, Global stability of discrete-time Cohen-Grossberg neural networks with impulses. \emph{Neurocomputing} \textbf{73} (2010) 3132–3138.

\bibitem{Porter09} M. Porter, J. Onnela and P. Mucha Communities in Networks \textit{AMS Notices} \textbf{56} (2009) 1082-1097, 1164-1166.

\bibitem{Strogatz03} S. Strogatz, Sync: The Emerging Science of Spontaneous Order, New York: Hyperion (2003).

\bibitem{LTao2011} L. Tao, W. Ting, and F. Shumin, Stability analysis on discrete-time Cohen-Grossberg neural networks with bounded distributed delay. Proceedings of the 30th Chinese Control Conference July 22-24, 2011, Yantai, China.

\bibitem{SChena2009} S. Chena, W. Zhaoa, and Y. Xub, New criteria for globally exponential stability of delayed Cohen-Grossberg neural network. \emph{Mathematics and Computers in Simulation} \textbf{79} (2009) 1527–1543.

\bibitem{LWang2005} L. Wang, Stability of Cohen-Grossberg neural networks with distributed delays. \emph{Applied Mathematics and Computation} \textbf{160} (2005) 93–110.

\bibitem{Watts99} D. Watts, Small Worlds: The Dynamics of Networks Between Order and Randomness, Princeton: Princeton Univ. Press, (1999).

\end{thebibliography}
\end{document}